\documentclass[11pt]{article}
\usepackage[T1]{fontenc}
\usepackage{lmodern}
\usepackage{amsmath,amsthm,amssymb,amsfonts}
\usepackage{enumerate}
\usepackage{epsfig}
\usepackage[dvipsnames,svgnames,table]{xcolor}
\usepackage{pgf,tikz}
\usetikzlibrary{calc}
\usepackage{fullpage}
\usepackage[colorlinks=true,linkcolor=RoyalBlue,urlcolor=RoyalBlue,citecolor=red]{hyperref}
\usepackage{booktabs}
\usepackage[nameinlink]{cleveref}
\usepackage{enumitem}
\usepackage{bm}
\usepackage{kbordermatrix}
\usepackage[title]{appendix}

\usepackage{thmtools, thm-restate}

\newtheorem{theorem}{Theorem}[section]
\newtheorem{proposition}[theorem]{Proposition}
\newtheorem{lemma}[theorem]{Lemma}
\newtheorem{corollary}[theorem]{Corollary}

\theoremstyle{definition}
\newtheorem{definition}[theorem]{Definition}
\newtheorem{example}[theorem]{Example}

\theoremstyle{remark}
\newtheorem{remark}[theorem]{Remark}

\newcommand\scalemath[2]{\scalebox{#1}{\mbox{\ensuremath{\displaystyle #2}}}}

\DeclareMathOperator{\rank}{rank}
\DeclareMathOperator{\im}{im}
\DeclareMathOperator{\Euc}{Euc}
\DeclareMathOperator{\Trans}{Trans}

\newcommand{\prob}{{\rm (P}_{G,\omega}{\rm )}}

\tikzstyle{labelsty}=[font=\scriptsize]
\tikzstyle{edge}=[line width=1.5pt]
\tikzstyle{dedge}=[edge,-latex]
\tikzstyle{redge}=[edge,dashed]
\tikzstyle{vertex}=[fill=black,circle,inner sep=0pt, minimum size=4pt]

\title{Uniquely realizable crystalline structures}
\author{Sean Dewar\thanks{School of Mathematics, University of Bristol. E-mail: \texttt{sean.dewar@bristol.ac.uk}} \qquad Bernd Schulze \thanks{School of Mathematical Sciences, Lancaster University. E-mail: \texttt{b.schulze@lancaster.ac.uk}}\qquad Shin-ichi Tanigawa \thanks{Department of Mathematical Informatics, University of Tokyo.
     E-mail: \texttt{tanigawa@mist.i.u-tokyo.ac.jp}} \qquad Louis Theran \thanks{School of Mathematics and Statistics, University of St Andrews. E-mail: \texttt{louis.theran@st-andrews.ac.uk}}
}

\begin{document}
\date{}
\maketitle

\begin{abstract}
	We construct infinite periodic versions of the stress matrix and establish sufficient conditions for periodic tensegrity frameworks  to be globally rigid  in $\mathbb{R}^d$  in the cases when the lattice is either fixed, fully flexible, or flexible with a volume constraint for the fundamental domain. For the fixed and fully flexible lattice variants, we also establish necessary and sufficient conditions for generic infinite periodic bar-joint frameworks to be globally rigid in $\mathbb{R}^d$. These results provide periodic versions of the fundamental results of  Connelly, as well as Gortler, Healy and Thurston on the global rigidity of generic finite bar-joint frameworks.
\end{abstract}



\section{Introduction}

\subsection{Background and motivation}

Global rigidity is concerned with determining when 
the edge lengths of a (bar-joint) framework determine its shape uniquely within a given space, up to congruence. This question has been extensively studied in the classical (finite) setting, where global rigidity theory is now well developed and remains a highly active area of research; see e.g. \cite[Chapter 63]{HandDCG},  \cite[Chapters 14, 16, 21]{HandMeera} and  \cite{grsum}  for summaries of recent results.  

The problem is motivated by a range of applications: in sensor network localization, one seeks to recover positions from distance data \cite{AspnesEGMWYAB2006}; in bio-chemistry and materials science, global rigidity is related to the stability of molecular and crystalline structures \cite{hermans2017rigidity,Whiteley2005Counting}. A central concept in global rigidity theory is an equilibrium stress, where the forces along the bars balance at each node. Such stresses are not only key to the theoretical characterizations of global rigidity but are also fundamental in practical applications, such as the design of efficient structural systems (e.g., gridshell roofs, tensegrities, or mechanical meta-materials).

Although much is known about the global rigidity of finite frameworks, the theory is far less developed for infinite periodic frameworks. Periodic  frameworks arise naturally in applications, as crystalline solids, biomolecular assemblies, and repetitive man-made structures all have inherently periodic geometries. 
This has motivated extensive research into their rigidity. See for example \cite{ww88,Ross2014,RossDCG,BS10,bsbull,MT13,tanigawa15}, as well as \cite[Chapter 62]{HandDCG} and \cite[Chapter 25]{HandMeera} for summaries of results. 

Since periodic frameworks have infinite underlying graphs, there are several possible definitions of local and global rigidity.
A widely studied one is the periodicity-forced rigidity introduced by Borcea and Streinu~\cite{BS10}, in which all deformations are required to preserve the given periodicity, that is, motions breaking the underlying lattice symmetry are not permitted.
Even within this periodicity-forced setting, the rigidity behavior of the framework depends on whether the representation of the underlying lattice is fixed or allowed to vary.
For example, the periodic framework shown on the right of \Cref{fig:hex} is flexible when the lattice is allowed to deform continuously,
since it admits a non-trivial motion that preserves both edge lengths and $\mathbb{Z}^2$-periodicity.
However, as we will see, the same framework is globally rigid when considered under a fixed-lattice representation, with the lattice given as on the left of \Cref{fig:hex}.

Over the past two decades,  substantial progress has been made in understanding periodicity-forced rigidity.
However, most of the papers listed above have focused on \emph{local rigidity} (i.e., rigidity within a neighborhood of the given framework).
In contrast, the \emph{global rigidity} of periodic frameworks has remained largely unresolved, with only a few recent works obtaining results -- specifically, combinatorial conditions for generic periodic rigidity under fixed-lattice representations (see \cite{kst21,kks22}). A central obstacle to further progress has been the lack of an appropriate analog of the stress matrix -- a fundamental tool in the theory of finite framework global rigidity introduced by Connelly \cite{Connelly1982}. As a result, periodic global rigidity has remained a long-standing open area of research.

\subsection{Our contributions}

In this paper, we introduce periodic versions of the stress matrix for the cases when the lattice is fully flexible, fixed, or has a lower bound for its volume. This allows a systematic analysis of global rigidity for periodic bar–joint frameworks and tensegrity structures in all of these settings  in arbitrary dimensions. (For simplicity of the exposition, throughout the paper we  restrict attention to fully $d$-periodic frameworks and tensegrities in $\mathbb{R}^d$, and do not consider $k$-periodic structures with $k<d$.)

Using our new stress matrices, we establish sufficient conditions for the global rigidity of periodic tensegrity frameworks for all the  lattice types under consideration (see \Cref{thm:suffpsd} for the fully flexible lattice, \Cref{thm:fixedpsd} for the fixed lattice,  and \Cref{thm:2}  for the lower-bounded volume lattice). Remarkably, in the lower-bounded volume case, the proof involves analyzing a non-convex optimization problem arising from the volume constraint, for which we show that every local minimizer is in fact a global minimizer.

Additionally, we prove the following main statement about generic periodic frameworks with a fully flexible lattice, where \Cref{t:flexible lattice sufficient} shows sufficiency  and \Cref{thm:ght}  shows necessity. This is the periodic analog of the theorems of Connelly \cite{conggr} and Gortler, Healy and Thurston \cite{gortler2010characterizing}.
Our result is stated in the language of \emph{$\mathbb{Z}^d$-gain graphs} (see \Cref{def:gain}), and all relevant framework terminology can be found in \Cref{sec:periodic} and \Cref{sec:equilibrium_stress}.

\begin{theorem}\label{mainthm:flexlattice}
    Let $(G,p,L)$ be a generic $\mathbb{Z}^d$-framework.
    Then the following are equivalent:
    \begin{enumerate}
        \item $(G,p,L)$ is globally rigid,
        \item $(G,p,L)$ is infinitesimally rigid and has an equilibrium stress $\omega$ where $\dim \ker \mathcal{L}_{\mathbb{Z}^d}(G,\omega) = d + 1$.
    \end{enumerate}
\end{theorem}

\Cref{mainthm:flexlattice} can be further simplified when $|V|=1$:
as every matrix $\mathcal{L}_{\mathbb{Z}^d}(G,\omega)$ has zero for each entry (\Cref{ex:flex1}),
we have that global rigidity and infinitesimal rigidity are equivalent.
We note here that this statement remains true even if $(G,p,L)$ has only trivial equilibrium stresses, indicating that each single vertex orbit periodic framework plays a similar role as the simplex in classical finite framework global rigidity.

The analogous fixed-lattice variant of \Cref{mainthm:flexlattice} also holds, with \Cref{t:fixed lattice sufficient} showing sufficiency and \Cref{thm:ghtfixed} showing necessity:

\begin{theorem}\label{mainthm:fixedlattice}
    Let $(G,p,L)$ be a generic $\mathbb{Z}^d$-framework.
    Then the following are equivalent:
    \begin{enumerate}
        \item $(G,p,L)$ is fixed-lattice globally rigid,
        \item $(G,p,L)$ has a fixed-lattice equilibrium stress $\omega$ where $\dim \ker \mathcal{L}(G,\omega) = 1$.
    \end{enumerate}
\end{theorem}

It follows immediately from \Cref{mainthm:flexlattice}  and \Cref{prop:generic stresses} (see \Cref{sec:suffgen}) that we obtain the following key result showing that global rigidity for periodic frameworks with a fully flexible lattice is in fact a generic property.

\begin{theorem}\label{thm:ggr}
    Let $G$ be a $\mathbb{Z}^d$-gain graph.
    Then either every generic $d$-dimensional realization of $G$ is globally rigid,
    or every generic $d$-dimensional realization of $G$ is not globally rigid.
\end{theorem}

Similarly, the combination of \Cref{mainthm:fixedlattice}  and \Cref{prop:fixed lattice generic stresses} (see \Cref{sec:suffgen}) shows that fixed-lattice global rigidity is also a generic property:

\begin{theorem}\label{thm:ggr fixedlattice}
    Let $G$ be a $\mathbb{Z}^d$-gain graph.
    Then either every generic $d$-dimensional realization of $G$ is fixed-lattice globally rigid,
    or every generic $d$-dimensional realization of $G$ is not fixed-lattice globally rigid.
\end{theorem}

For the lower-bounded volume lattice, it turns out that global rigidity is  not a generic property, as explained in \Cref{sec:boundedlat}, and hence we have no analogous result to \Cref{mainthm:flexlattice} or \Cref{mainthm:fixedlattice} in that case. 

\subsection{Structure of the paper}

We begin by establishing the necessary background and notation for periodic graphs and their group-labeled quotient (or ``gain'') graphs  in \Cref{sec:periodicgraphs}. In \Cref{sec:periodic}, we then define global rigidity for periodic frameworks under fixed and 
flexible lattice representations, and review key notions from the corresponding local rigidity theory. In \Cref{sec:equilibrium_stress} we introduce  stress matrices for periodic frameworks under fixed and flexible lattice representations, and establish some basic properties of these new matrices. \Cref{sec:suff} is concerned with sufficient conditions (given in terms of the new stress matrices) for global rigidity of tensegrity frameworks for both the fixed and the flexible lattice. Sufficient and necessary conditions for \emph{generic} periodic frameworks to be globally rigid are established in \Cref{sec:suffgen} and \Cref{sec:necessgen}, respectively. Finally, in \Cref{sec:boundedlat} we  give a sufficient condition for a periodic tensegrity to be globally rigid under a flexible lattice representation with a volume constraint.

\section{\texorpdfstring{$\mathbb{Z}^d$}{Zd}-gain graphs and periodic graphs}\label{sec:periodicgraphs}

\subsection{Basic definitions}

Fix $V$ to be a finite set and $d$ to be some positive integer.
Define $\sim$ to be the equivalence relation on the set $V \times V \times \mathbb{Z}^d$ where $(u,v,\gamma) \sim (u',v',\gamma')$ if and only if $u'=u$, $v'=v$ and $\gamma' = \gamma$, or $u'=v$, $v'=u$ and $\gamma' = -\gamma$.
It is immediate that each equivalence class either contains exactly two elements or is of the form $\{(u,u,0)\}$ for some $u \in V$ (later on this special case will also be discounted).
By abuse of notation,
each equivalence class of $V \times V \times \mathbb{Z}^d/ \sim$ is denoted by one of the two possible representations for it;
this representation will for the most part be chosen to always be optimal for us (this will become clear later on).

\begin{definition}\label{def:gain}
    A \emph{$\mathbb{Z}^d$-gain graph} (sometimes also called a $\mathbb{Z}^d$-labeled graph) is a pair $G=(V,E)$,
    where $V$ is a finite set, $E^+ \subset V \times V \times \mathbb{Z}^d$ is a finite subset where $E = E^+/\sim$, and for any element $e = (u,v,\gamma) \in E^+$ (which we now say is an \emph{edge} of $G$) we have $\gamma \neq 0$ if $u=v$, and $(v,u,-\gamma) \in E^+$.
\end{definition}

\begin{remark}
    Our slightly strange equivalence relation on the set $V \times V \times \mathbb{Z}^d$ now allows us to ``flip'' any edge we wish to be in whatever orientation we require.
    Note that we now cannot have two edges with the same orientation and same gains,
    or with opposite orientation and gains that sum to zero (i.e., $(u,v,\gamma)$ and $(v,u, -\gamma)$ cannot both occur since $E$ is not a multiset and the two edges are equivalent under $\sim$).
\end{remark}

For an edge $e=(u,v,\gamma)$,
we define the element $\gamma$ to be the \emph{gain} of $e$.
If $u=v$ then we say $e$ is a \emph{loop}.
For any vertex $x \in V$,
we define the set of non-loop edges at $x$ by
\begin{align*}
    \delta_G(x) := \{ (u,x,\gamma) : (u,x,\gamma) \in E, ~ x \neq u \}.
\end{align*}
We define a $\mathbb{Z}^d$-gain graph $H$ to be a \emph{subgraph} of $G$ if its vertex set (denoted $V(H)$) is contained in $V$ and its edge set (denoted $E(H)$) is contained in $E$.
On a similar note, any terminology surrounding finite directed multigraphs (e.g., connectivity) can be directly applied to $\mathbb{Z}^d$-gain graphs by ``forgetting'' the gain assigned to each edge.

The  \emph{covering graph}  (or \emph{derived graph}) of a $\mathbb{Z}^d$-gain graph $G$ is the (undirected, simple) graph $\tilde{G}=(\tilde{V},\tilde{E})$ with vertex set  $\tilde{V} =V\times \mathbb{Z}^d$ and edge set $\tilde{E}=E \times \mathbb{Z}^d$,
where two vertices $(u,\alpha),(v,\beta)$ of $\tilde{G}$ are adjacent if and only if $(u,v,\beta-\alpha) \in E$ or $(v,u,\alpha-\beta) \in E$.
The $\mathbb{Z}^d$-gain graph $G$ is also called the \emph{quotient $\mathbb{Z}^d$-gain graph} of $\tilde{G}$, and $\tilde{G}$ is called a \emph{$d$-periodic graph}.
Note that this is well-defined since the covering graph $\tilde{G}$ is identical no matter which representative we choose for each edge of $G$. 
See  \cite{BS10,rossthesis} for more details. (See also \cite{jkt2016}, for example, for the corresponding notions for finite symmetric graphs.)

\begin{remark}\label{rem:switching}
    We can alternatively define a simple undirected graph $\tilde{G}=(\tilde{V},\tilde{E})$ to be a $d$-periodic graph if the group $\mathbb{Z}^d$ acts freely on $\tilde{G}=(\tilde{V},\tilde{E})$,
    and the action produces only finitely many vertex (respectively, edge) orbits.
    From this we see that the quotient $\mathbb{Z}^d$-gain graph of $G$ is formed from the quotient graph of $\tilde{G}/\mathbb{Z}^d$,
    with an assignment of gains to each edge orbit so as to be able to reconstruct $\tilde{G}$.
    While we do have an infinite choice of possible ways of choosing this assignment of gains for $\tilde{G}/\mathbb{Z}^d$,
    they are all equivalent under so-called switching operations;
    see \cite{rossthesis} for more details.
    Furthermore,
    the rank of the incidence matrix of a $\mathbb{Z}^d$-gain graph is invariant under switching operations.
\end{remark}

\subsection{Subgraph rank}

A \emph{walk} in a $\mathbb{Z}^d$-gain graph $G=(V,E)$ is a finite sequence of edges $W =(e_1,\ldots,e_t)$, each of the form $e_i = (u_i,v_i,\gamma_i)$,
such that $u_{i+1} = v_i$ for each $1 \leq i \leq t-1$.
We say that the walk $W$ \emph{passes through $v$} if $v$ is one of the ends of any edge contained in $W$.
If $v_t=u_1$ then $W$ is said to be a \emph{closed walk}.
Using our fixed orientation for each edge in the walk,
we define the \emph{gain map}
\begin{align*}
    z(W) :=\sum_{i=1}^t \gamma_i.
\end{align*}
For a subgraph $H$ of $G$ and $v\in V(H)$, we let $\langle H \rangle_{v}$ be the subgroup of $\mathbb{Z}^d$ defined by 
\begin{align*}
    \langle H\rangle_{v} = \left\{z(W) : W \text{ is a closed walk in } H \text{ passing through } v \right\}.
\end{align*}
If $u$ and $v$ are in the same connected component of $H$ of $G$, then the groups $\langle H\rangle_u$ and 
$\langle H\rangle_v$ are conjugate in $\mathbb{Z}^d$.  Since all subgroups of $\mathbb{Z}^d$ are normal, 
this implies that $\langle H\rangle_u = \langle H\rangle_v$.  So, for a connected graph $G$, 
we may write simply $\langle G\rangle$.  

\begin{definition}
    Let $H$ be a subgraph of the $\mathbb{Z}^d$-gain graph $G$ (here allowing the special case of $H=G$).
    If $H$ is connected then the \emph{rank} of $H$ (denoted by $\rank H$) is defined to be the rank of the free group $\langle H\rangle$, i.e., the smallest cardinality of a generating set for $\langle H\rangle$.
    If $H$ is disconnected with connected components $H_1,\ldots, H_c$ then the rank of $H$ is defined to be
    \begin{align*}
        \rank H := \max_{i\in \{1,\ldots, c\}} \rank H_i.
    \end{align*}
\end{definition}

\subsection{Incidence matrices}\label{subsec:incidence}

We now fix an arbitrary orientation for each edge of our $\mathbb{Z}^d$-gain graph $G=(V,E)$.
An \emph{incidence matrix} $I(G)$ of $G$ is defined to be a matrix of size $|E|\times |V|$ such that each row corresponds to an edge of $G$ and 
each column corresponds to a vertex of $G$.
More specifically,
the row of $I(G)$ associated with an edge $e=(u,v, \gamma)\in E$, $u \neq v$, is of the form
\begin{align*}
    \kbordermatrix{
&& & & u&&&&v& & & \\
(u,v,\gamma)&0& \dots& 0& -1 &0&\dots&0& 1 & 0& \dots&0},
\end{align*}
while a loop ($u=v$) gives an all zero row.
Note that our choice of orientation for each edge of $G$ corresponds to multiplying rows by $\pm 1$.
The \emph{$\mathbb{Z}^d$-incidence matrix} of $G$ is the $|E| \times (|V|+d)$ matrix $I_{\mathbb{Z}^d}(G)$ 
such that each row corresponds to an edge of $G$, 
each of the first $|V|$ columns corresponds to a vertex of $G$, and each of the last $d$ columns corresponds to a coordinate of the labeling space $\mathbb{Z}^d$. 
More specifically,
the row of $I_{\mathbb{Z}^d}(G)$ associated with an edge $e=(u,v, \gamma)\in E$, $u\neq v$,  is of the form
\begin{align*}
    \kbordermatrix{
&& & & u&&&&v& & && \mathbb{Z}^d \\
(u,v,\gamma)&0& \dots& 0& -1 &0&\dots&0& 1 & 0& \dots&0& \gamma^{\top}},
\end{align*}
while we have zero for the column corresponding to vertex $v$ if $u=v$.
Alternatively,
if we fix $M(G)$ to be the $d \times |E|$ matrix where column $e=(u,v,\gamma)$ is given by the vector $\gamma$,
we have
\begin{equation*}
    I_{\mathbb{Z}^d}(G) = \left[ I(G) ~ ~ M(G)^\top \right].
\end{equation*}

We now argue that $I_{\mathbb{Z}^d}(G)$ is a natural $\mathbb{Z}^d$-gain graph analog of ordinary incidence matrices of graphs.
Fix the column vector $\hat{\mathbf{1}}:=(1,\dots, 1,0\dots, 0)^{\top}\in\mathbb{R}^{|V|+d}$, i.e., $\hat{\mathbf{1}}$ is the concatenation of the $|V|$-dimensional all-ones vector and the $d$-dimensional zero vector.
The kernel of $I_{\mathbb{Z}^d}(G)$ always contains $\hat{\mathbf{1}}$,
and hence $\rank I_{\mathbb{Z}^d}(G) \leq |V|-1+d$ holds.  
For $d\leq 2$, Malestein and Theran observed that the rank of $I_{\mathbb{Z}^d}(G)$ can be characterized by a simple combinatorial condition; see \cite[Lemma 2.5]{MT13}.
We now provide a proof of this statement for higher dimension gain graphs.

\begin{proposition}\label{prop:1}
    Let $G$ be a $\mathbb{Z}^d$-gain graph.
    Then the following are equivalent.
    \begin{enumerate}
        \item \label{prop:1:item1} $G$ is connected (equivalently, $\rank I(G) = |V| - 1$) and the rank of $G$ is $d$.
        \item \label{prop:1:item2} $\rank I_{\mathbb{Z}^d}(G) =|V|-1+d$.
    \end{enumerate}
\end{proposition}
\begin{proof}
    It is immediate that $\rank I_{\mathbb{Z}^d}(G) \leq \rank I(G) + \rank M(G)^\top$,
    and hence \ref{prop:1:item2} implies \ref{prop:1:item1}.
    Now suppose that \ref{prop:1:item1} holds.
    The group $\langle G\rangle$ is known to be generated by the images under the gain map $z$ of any basis for the homology group $H_1(G;\mathbb{Z})$ (here with $G$ being considered solely as a multigraph).  
    Since $G$ is connected, there is such a 
    basis of the following form: pick a spanning tree $T$ of $G$; for each edge $e\in E\setminus T$, 
    there is a fundamental cycle $C_e$ with respect to $T$. 
    We now relabel $G$ so that an edge $e$ has gain $0$ if $e \in T$ and $e$ has gain $z(C_e)$ if $e \in E \setminus T$.
    Let us call the new gain graph $G'$.
    By the previous fact, $z(W) = z(W')$ for every pair of 
    corresponding closed walks in $G$ and $G'$.  Hence the row matroids of $I_{\mathbb{Z}^d}(G)$
    and $I_{\mathbb{Z}^d}(G')$ are isomorphic.  We thus have (after maybe permuting rows)
    \[
        I_{\mathbb{Z}^d}(G') = 
        \begin{bmatrix}
            I(T) & \mathbf{0}_{|T| \times d} \\
            I(G\setminus T) & X^\top
        \end{bmatrix}
    \]
    where the rows of $X^\top$ generate the group $\langle G' \rangle \cong \mathbb{Z}^d$.
    Hence, the rank of $I_{\mathbb{Z}^d}(G')$ is at least $\rank I(G) + \rank(M') = n - 1 + d$.  This shows \ref{prop:1:item2}.
\end{proof}

\section{Periodic frameworks}\label{sec:periodic}

For this section we fix a $d$-periodic graph $\tilde{G}=(\tilde{V},\tilde{E})$ and its quotient $\mathbb{Z}^d$-gain graph $G=(V,E)$.
Importantly,
this implies every vertex of $\tilde{G}$ is of the form $(v,\gamma)$ for some $v \in V$ and some $\gamma \in \mathbb{Z}^d$.

\subsection{Global rigidity of periodic frameworks}

Let $L$ be a $d \times d$ (real-valued) matrix.
(In the following, we will often refer to any $d \times d$ matrix as a \emph{lattice}.)
An \emph{$L$-periodic framework} in $\mathbb{R}^d$ is a pair $(\tilde{G},\tilde{p})$ of a simple graph $\tilde{G}$ and a realization $\tilde{p}: \tilde{V} \rightarrow \mathbb{R}^d$ such that for each vertex $(v, \alpha) \in \tilde{V}$ and every integer vector $\gamma \in \mathbb{Z}^d$ we have
\begin{equation*}
    \tilde{p}((v,\alpha))+L\gamma = \tilde{p}((v,\alpha+\gamma)).
\end{equation*}
We say that a framework $(\tilde{G}, \tilde{p})$ is \emph{$d$-periodic} in $\mathbb{R}^d$ if it is an $L$-periodic framework for some lattice $L$.
A $d$-periodic framework is \emph{non-flat} if the lattice $L$ is nonsingular,
and \emph{affinely spanning} if the affine span of the image of $\tilde{p}$ is $\mathbb{R}^d$.
Although every non-flat $d$-periodic framework is affinely spanning, the converse is not always true.

We now define global rigidity for an $L$-periodic framework in the analogous way as for finite frameworks \cite{kst21}. 
An $L$-periodic framework $(\tilde{G},\tilde{p})$ and an $L'$-periodic framework $(\tilde{G},\tilde{q})$ in $\mathbb{R}^d$ are said to be \emph{equivalent} if
\begin{equation*}
\label{eq:periodic_equiv}
\left\| \tilde p(x)- \tilde p(y)\right\|=\left\| \tilde q(x)-\tilde q(y)\right\|\qquad \text{for all } xy\in \tilde{E},
\end{equation*}
and they are said to be \emph{congruent} if
\begin{equation*}
\label{eq:periodic_congr}
\left\| \tilde p(x)- \tilde p(y)\right\|=\left\| \tilde q(x)-\tilde q(y)\right\|\qquad \text{for all } x,y\in \tilde{V}.
\end{equation*}
An $L$-periodic framework $(\tilde{G},\tilde{p})$ is called \emph{globally rigid} if every $d$-periodic framework $(\tilde{G},\tilde{q})$ in $\mathbb{R}^d$ which is equivalent to $(\tilde{G},\tilde{p})$ is also congruent to $(\tilde{G},\tilde{p})$.
Occasionally, we will wish to use the following weaker variant of global rigidity:
an $L$-periodic framework $(\tilde{G},\tilde{p})$ is called \emph{globally $L$-rigid} if every $L$-periodic framework $(\tilde{G},\tilde{q})$ in $\mathbb{R}^d$ which is equivalent to $(\tilde{G},\tilde{p})$ is also congruent to $(\tilde{G},\tilde{p})$ (note that the second framework now has an identical lattice to the first).

\subsection{Global rigidity of \texorpdfstring{$\mathbb{Z}^d$}{d}- frameworks}

It is convenient to analyse the global rigidity of periodic frameworks in terms of their quotient structure.
Let $(\tilde{G}, \tilde{p})$ be an $L$-periodic framework.
We define the \emph{quotient $\mathbb{Z}^d$-framework} as the triple $(G, p, L)$ with $p: V \rightarrow \mathbb{R}^d$ by setting $p(v) := \tilde{p}(v,0)$ for each $v \in V$.

In general, a \emph{$\mathbb{Z}^d$-framework} (or simply \emph{framework}) is defined to be a triple $(G, p, L)$ of a $\mathbb{Z}^d$-gain graph $G$, a map $p:V \rightarrow \mathbb{R}^d$, and a lattice $L \in \mathbb{R}^{d \times d}$.
The pair $(p,L)$ are said to be a \emph{realization} of $G$.
The \emph{covering} of $(G,p,L)$ is the $d$-periodic framework $(\tilde{G}, \tilde{p})$,
where $\tilde{G}$ is the covering graph of $G$ and $\tilde{p} :\tilde{V} \rightarrow \mathbb{R}^d$ is the uniquely determined realization where $\tilde{p}(v,\gamma) = p(v) +L\gamma$ for each vertex $(v,\gamma) \in \tilde{V}$.
We say that a $\mathbb{Z}^d$-framework is \emph{non-flat} (respectively, \emph{affinely spanning}) if its covering is non-flat (respectively, affinely spanning).
Equivalently,
a $\mathbb{Z}^d$-framework $(G,p,L)$ is: (i) non-flat if and only if $L$ is nonsingular,
and (ii) affinely spanning if and only if the set $\{p(v) + L\mu : v \in V, ~ \mu \in \mathbb{Z}^d\}$ is affinely spanning.

The set of all possible pairs of $(p,L)$ is called the \emph{realization space} ${\cal R}(G)$ of a $\mathbb{Z}^d$-gain graph $G$.
The subspace of non-flat realizations is denoted by ${\cal R}^*(G)$.
We denote the set of all possible maps $p: V \rightarrow \mathbb{R}^d$ by $(\mathbb{R}^d)^V$.

We define the \emph{measurement map} of $G$ to be the map
\begin{align*}
    f_G: \mathcal{R}(G) \rightarrow \mathbb{R}^E, ~ (p,L) \mapsto \left( \|p(v) + L \gamma - p(u)\|^2 \right)_{(u,v,\gamma) \in E}.
\end{align*}
Observe that if $(q,L') \in \mathcal{R}(G)$ satisfies $f_G(p,L) = f_G(q,L')$ for some $\mathbb{Z}^d$-framework $(G,p,L)$,
then $(G,q,L')$ is also a $\mathbb{Z}^d$-framework.

We say that two $\mathbb{Z}^d$-frameworks $(G,p,L)$ and $(G,q,L')$ are \emph{equivalent} if $f_G(p,L) = f_G(q,L')$.
Furthermore,
if there exists an orthogonal $d \times d$ matrix $M$ and a vector $x \in \mathbb{R}^d$ such that $q(v) = M p(v) +x$ and $L' = M L$ then we say that $(G,p,L)$ and $(G,q,L')$ are \emph{congruent}.
With these definitions we define the following rigidity concepts.

\begin{definition}
    We say a $\mathbb{Z}^d$-framework $(G,p,L)$ is:
    \begin{enumerate}
        \item \emph{locally rigid} there exists $\varepsilon > 0$ such that every $\mathbb{Z}^d$-framework $(G,q,L')$ with $\|q(v) - p(v)\|< \varepsilon$ and $\|L'-L\| <\varepsilon$\footnote{Here we can use whatever norm is preferred; for example, the operator norm.} that is equivalent to $(G,p,L)$ is congruent to $(G,p,L)$; and
        \item \emph{globally rigid} if every $\mathbb{Z}^d$-framework that is equivalent to $(G,p,L)$ is congruent to $(G,p,L)$.
    \end{enumerate}
\end{definition}

We likewise define the following fixed-lattice variants.

\begin{definition}
    We say a $\mathbb{Z}^d$-framework $(G,p,L)$ is:
    \begin{enumerate}
        \item \emph{fixed-lattice locally rigid} there exists $\varepsilon > 0$ such that every $\mathbb{Z}^d$-framework $(G,q,L)$ with $\|q(v) - p(v)\|< \varepsilon$ that is equivalent to $(G,p,L)$ is congruent to $(G,p,L)$; and
        \item \emph{fixed-lattice globally rigid} if every $\mathbb{Z}^d$-framework with lattice $L$ that is equivalent to $(G,p,L)$ is congruent to $(G,p,L)$.
    \end{enumerate}
\end{definition}

One readily verifies
that an $L$-periodic framework $(\tilde{G},\tilde{p})$ is globally rigid (respectively, globally $L$-rigid) if and only if its quotient $\mathbb{Z}^d$-framework $(G,p,L)$ is globally rigid (respectively, fixed-lattice globally rigid). 

\begin{proposition}\label{prop:defgr}
An $L$-periodic framework $(\tilde{G},\tilde p)$ is globally rigid if and only if its quotient $\mathbb{Z}^d$-framework  $(G,p,L)$ is globally rigid.
\end{proposition}
\begin{proof}
 Suppose $(\tilde{G},\tilde p)$ is globally rigid. 
Consider its quotient $\mathbb{Z}^d$-framework $(G,p,L)$. Let $(G,q,L')$ be a $\mathbb{Z}^d$-framework with $f_G(p,L) = f_G(q,L')$.  
By periodicity, we have that $(\tilde{G},\tilde{p})$ is equivalent to $ (\tilde{G},\tilde{q})$.  
By assumption, $(\tilde{G},\tilde{p})$ is then congruent to $ (\tilde{G},\tilde{q})$, so \emph{all} vertex pair distances agree.  
It is well known that this then implies there exists an isometry $(M,x)$ such that
\[
\tilde{q}(v) = M \tilde{p}(v) + x \text{ for all } v \in \tilde{V}, \qquad L' = M L.
\]
This now implies $(G,p,L)$ and $(G,q,L')$ are congruent with the same $M$ and $x$.

Suppose now that $(G,p,L)$ is globally rigid. 
Consider the $L$-periodic covering  framework $(\tilde{G},\tilde{p})$.  
Let $(\tilde{G},\tilde{q})$ be an $L'$-periodic framework that is equivalent to $(\tilde{G},\tilde{p})$.  
Then, in particular we have $f_G(p,L) = f_G(q,L')$.  
By assumption, there exists an isometry $(M,x)$ such that
\[
q(v) = M p(v) + x \text{ for all } v \in V, \qquad L' = M L.
\]
From this, it follows that 
\begin{equation*}
\left\| \tilde q(x)- \tilde q(y)\right\|=\left\| (M\tilde p(x) + x)- (M\tilde p(y) + x)\right\|=\left\| \tilde p(x)-\tilde p(y)\right\|\qquad \text{for all } x,y\in \tilde{V}.
\end{equation*}
and so $(\tilde{G},\tilde{p})$ is congruent to $ (\tilde{G},\tilde{q})$.
\end{proof}

An analogous proof applies to the fixed-lattice version of \Cref{prop:defgr}:
\begin{proposition}
An $L$-periodic framework $(\tilde{G},\tilde p)$ is  globally $L$-rigid if and only if its quotient $\mathbb{Z}^d$-framework  $(G,p,L)$ is fixed-lattice globally rigid.
\end{proposition}

When dealing with fixed-lattice global rigidity,
we often benefit from considering only the point-configuration component $p$.
For a chosen lattice $L$, we define the \emph{$L$-measurement map} of $G$ to be the map 
\begin{equation*}
    f_{G,L}: (\mathbb{R}^d)^V \rightarrow \mathbb{R}^E, ~ p \mapsto f_G(p,L) = \left( \|p(v) + L \gamma - p(u)\|^2 \right)_{(u,v,\gamma) \in E}.
\end{equation*}
It now follows that a non-flat $\mathbb{Z}^d$-framework $(G,p,L)$ is fixed-lattice globally rigid if and only if for every $q \in (\mathbb{R}^d)^V$ where $f_{G,L}(p) = f_{G,L}(q)$,
there exists a vector $x \in \mathbb{R}^d$ such that $q(v) = p(v)+x$ for all $v \in V$.

\subsection{Coordinate representations and genericity}\label{subsec:coord}

It is sometimes convenient to identify  ${\cal R}(G)$ with $\mathbb{R}^{dn+d^2}$,
where $n = |V|$.
To avoid confusion, throughout the paper we do the identification through the following rule.
First,
we label the vertices of $G$ as $V = \{v_1,\ldots,v_n\}$.
Next,
we consider a map $p \in (\mathbb{R}^d)^V$ as a column vector in $\mathbb{R}^{dn}$ by concatenating the $d$-dimensional vectors $p(v_1),\ldots,p(v_n)$ according to our pre-determined ordering of $V$.
This vector is now said to be the \emph{point-configuration vector},
which we denote, by abuse of notation, as $p$.
For a lattice $L \in \mathbb{R}^{d \times d}$, we denote the $i$-th column vector of $L$ by $\ell_i$, that is, $L= [\ell_1 ~ \cdots ~ \ell_d ]$. 
The \emph{lattice vector} $\ell\in \mathbb{R}^{d^2}$ is now obtained by concatenating $\ell_1, \dots, \ell_d$.
We now define the concatenation 
\begin{align*}
    \begin{bmatrix} 
        p \\ 
        \ell
    \end{bmatrix} = [ p^\top ~ \ell^\top]^\top
    \in \mathbb{R}^{dn+d^2}
\end{align*}
to be the \emph{vector representation} of $(p,L)\in {\cal R}(G)$.

A $\mathbb{Z}^d$-framework $(G,p,L)$ is said to be \emph{generic} if the coordinates of the vector representation of $(p,L)$ are algebraically independent over the rationals.
A $\mathbb{Z}^d$-framework $(G,p,L)$ is said to be \emph{$L$-generic} if the coordinates of the point-configuration vector $p$ are algebraically independent over the field generated by the rationals and the entries of $L$.
It is immediate that every generic framework $(G,p,L)$ is also $L$-generic;
however the reverse is not true in general.

\subsection{The rigidity matrix}\label{subsec:rigmat}
The derivative of $f_G$ evaluated at $(p,L)$, denoted by $df_G(p,L)$, is the linear map
\begin{align*}
    df_G(p,L) : \mathcal{R}(G) \rightarrow \mathbb{R}^E, ~ (m,M) \mapsto \Big( 2 (p(v) + L \gamma - p(u) ) \cdot  (m(v) + M\gamma - m(u) ) \Big)_{(u,v,\gamma) \in E}.
\end{align*}
Any element of $\ker df_{G}(p,L)$ is said to be an \emph{infinitesimal motion} of $(G,p,L)$.
For any skew-symmetric $d \times d$ matrix $A$ (i.e., an element of the tangent space of orthogonal matrices at $I_d$) and for any vector $z \in \mathbb{R}^d$,
we can construct an infinitesimal motion $(m,M)$   by setting $m(v) = Ap(v) +x$ for each $v \in V$ and $M= AL$;
any such infinitesimal motion is said to be \emph{trivial}.
The trivial infinitesimal motions of a $\mathbb{Z}^d$-framework correspond to the rigid body motions in $\mathbb{R}^d$.

The derivative $df_G(p,L)$ plays an important role in the rigidity analysis of a framework $(G,p,L)$.
Since the simplified map $\frac{1}{2}df_G(p,L)$ is a real linear map,
we naturally wish to also be able to consider it as a real valued matrix.
More specifically, the \emph{rigidity matrix} $R(G,p,L)$ of $(G,p,L)$ is the matrix of size $|E|\times (d|V|+d^2)$ where the row of $R(G,p,L)$ associated with an edge $e=(u,v,\gamma)$, $u\neq v$, is of the form 
\begin{align*}
    \kbordermatrix{
    && & & u&&&&v& & && \ell \\
    e = (u,v,\gamma) &0& \dots& 0& -\nu (e)^{\top}&0&\dots&0& \nu(e)^{\top}& 0& \dots&0& (\gamma \otimes \nu(e))^{\top}},
\end{align*}
where  $\nu (e) :=p(v) + L\gamma -p(u)$ and, given $\gamma = (\gamma_1,\ldots,\gamma_d)$,
\begin{align*}
    \gamma \otimes \nu(e) := \begin{bmatrix} \gamma_1 \nu(e)\\ \vdots \\ \gamma_d \nu(e)\end{bmatrix}\in \mathbb{R}^{d^2}.
\end{align*}
If the edge $e$ is a loop, i.e., $u=v$, then the row of $R(G,p,L)$ associated with an edge $e=(v,v,\gamma)$, is of the form 
\begin{align*}
    \kbordermatrix{
    &&& &v& & && \ell \\
    e = (u,v,\gamma) &0& \dots&0& \mathbf{0}^\top& 0& \dots&0& (\gamma \otimes \nu(e))^{\top}},
\end{align*}
where $\mathbf{0} \in \mathbb{R}^d$ is the all-zeroes vector.

From our construction we note that a pair $(m,M) \in \mathcal{R}(G)$ is an element of the kernel of $df_G(p,L)$ if and only if the vector representation of $(m,M)$ (see \Cref{subsec:coord}) is an element of the kernel of $R(G,p,L)$.

The rigidity matrix $R(G,p,L)$ was first introduced by Borcea and Streinu in \cite{BS10}, and it is now a fundamental tool in the theory of periodic frameworks.
Since the trivial infinitesimal motions corresponding to the rigid body motions in $\mathbb{R}^d$ are always elements of the kernel of $df_G(p,L)$, 
we have (assuming $(G,p,L)$ is affinely spanning) that nullity of $R(G,p,L)\geq \binom{d+1}{2}$. 
An affinely spanning $\mathbb{Z}^d$-framework $(G,p,L)$ is called \emph{infinitesimally rigid} if the nullity of $R(G,p,L)$ is $\binom{d+1}{2}$, or equivalently, $\textrm{rank } R(G,p,L)=dn+d^2-\binom{d+1}{2}= dn+\binom{d}{2}$.

The concepts of infinitesimal rigidity and local rigidity are closely linked.
The following can be proven using similar techniques to those employed by Asimow and Roth \cite{AsimowRothI,AsimowRothII};
see for example Appendix A in the 2010 arxiv version of \cite{MT13}.

\begin{theorem}\label{t:gen rig equiv}
    Infinitesimal rigidity implies local rigidity for all $\mathbb{Z}^d$-frameworks.
    Additionally, 
    local rigidity implies infinitesimal rigidity for all generic $\mathbb{Z}^d$-frameworks.
\end{theorem}

It is well known that if a $\mathbb{Z}^d$-framework $(G,p,L)$ is infinitesimally rigid, then so too is every generic $\mathbb{Z}^d$-framework $(G,p',L')$.
When $p$ is generic and $d=2$, a combinatorial characterization of the rank of $R(G,p, L)$ was established by Malestein and Theran in \cite{MT13}.

\subsection{The fixed-lattice rigidity matrix}

The derivative of $f_{G,L}$ evaluated at $p$, denoted by $df_{G,L}(p)$, is the linear map
\begin{align*}
    df_{G,L}(p) : (\mathbb{R}^d)^V \rightarrow \mathbb{R}^E, ~ u \mapsto \Big( 2 (p(v) + L \gamma - p(u) ) \cdot  (u(v) - u(u) ) \Big)_{(u,v,\gamma) \in E}.
\end{align*}
As in the previous section,
it is advantageous to consider the simplified map $\frac{1}{2}df_{G,L}(p)$ as a real valued matrix.
The \emph{$L$-rigidity matrix} $R_L(G,p)$ of $(G,p,L)$ is the matrix of size $|E|\times d|V|$ where the row of $R(G,p,L)$ associated with the edge $e=(u,v,\gamma)$ is of the form
\begin{align*}
    \kbordermatrix{
    && & & u&&&&v& & && \\
    e = (u,v,\gamma) &0& \dots& 0& -(p(v) + L\gamma -p(u))^{\top}&0&\dots&0& (p(v) + L\gamma -p(u))^{\top}& 0& \dots&0& }.
\end{align*}
From our construction we note that $m \in (\mathbb{R}^d)^V$ is an element of the kernel of $df_{G,L}(p)$ if and only if the point-configuration vector form of $m$ (see \Cref{subsec:coord}) is an element of the kernel of $R_L(G,p)$.
From this correspondence we see that the kernel of $R_L(G,p)$ is isomorphic to the space of ``infinitesimal motions'' of $(G,p,L)$ that do not ``infinitesimally deform'' its lattice.

Again, fixed-lattice infinitesimal rigidity and fixed-lattice local rigidity are closely linked.
A proof of the following can be found in  \cite[Section 3.3]{rossthesis}.

\begin{theorem}\label{t:gen rig equiv fl}
    Fixed-lattice infinitesimal rigidity implies fixed-lattice local rigidity for all $\mathbb{Z}^d$-frameworks.
    Additionally, for any lattice $L$, 
    fixed-lattice local rigidity implies fixed-lattice infinitesimal rigidity for all $L$-generic $\mathbb{Z}^d$-frameworks.
\end{theorem}

Fixed-lattice infinitesimal (and local) rigidity are $L$-generic properties of the $\mathbb{Z}^d$-gain graph, and a combinatorial characterization for $d=2$ is given in \cite{RossDCG}.

\begin{remark}
    Similarly to (fixed-lattice) global rigidity, one can also define infinitesimal and local (fixed-lattice) rigidity for the covering framework;
    see, for example, \cite{BS10,MT13,rossthesis}.
    Since we focus on the quotient frameworks whenever this is needed, we omit these definitions.
\end{remark}


\section{Equilibrium stresses of periodic frameworks}\label{sec:equilibrium_stress}

Let $(G,p,L)$ be a $\mathbb{Z}^d$-framework and let $\omega:E \rightarrow \mathbb{R}$ be an edge weighting.
We say that $\omega$ is a \emph{fixed-lattice equilibrium stress of $(G,p,L)$} if for every vertex $x \in V$ we have
\begin{equation}\label{eq:1}
    \sum_{e = (u,x,\gamma) \in \delta_G(x)} \omega(e) (p(x) + L\gamma -p(u)) = \mathbf{0},
\end{equation}
where $\mathbf{0} \in \mathbb{R}^d$ is the all-zeroes vector and $\delta_G(x)$ denotes the set of non-loop incident edges at $x$ as described in \Cref{sec:periodic}.
If $\omega:E \rightarrow \mathbb{R}$ is an edge weighting that satisfies \cref{eq:1} and 
\begin{equation}\label{eq:2}
\sum_{e = (u,v,\gamma)\in E} \omega(e) (p(v) + L\gamma -p(u)) \gamma^{\top}=\mathbf{0}_{d \times d},
\end{equation}
where $\mathbf{0}_{d \times d} \in \mathbb{R}^{d \times d}$ is the all-zeroes matrix, 
then we say $\omega$ is an \emph{equilibrium stress of $(G,p,L)$}.
(Here each term of the summation is a $d\times d$ matrix.)

\begin{remark}
Since \cref{eq:1} is equivalent to the equilibrium condition for the covering, in \cite{BorStreinuDCG15}, Borcea and Streinu referred to an $\omega$ satisfying \cref{eq:1} as an ``(ordinary) self-stress'', and an $\omega$ satisfying both \cref{eq:1} and \cref{eq:2} as a ``periodic stress''.
We adopt the terminology ``fixed-lattice equilibrium stress'' to make our discussion clearer, since we will deal with both the flexible-lattice and the fixed-lattice model in this paper.
\end{remark}

Note that \cref{eq:2} holds if and only if 
\begin{equation}\label{eq:2.1}
\sum_{e = (u,v,\gamma)\in E} \omega(e) (p(v) + L\gamma -p(u)) \otimes \gamma = \mathbf{0}
\end{equation}
where $\mathbf{0} \in \mathbb{R}^{d^2}$ is the all-zeroes vector.
Thus, an edge weight $\omega$ may be defined to be an equilibrium stress if it satisfies $\omega^{\top} R(G,p,L)=\mathbf{0}^\top$, and the space of equilibrium stresses is just the left kernel of $R(G,p,L)$ \cite{BS10}. Similarly, fixed-lattice equilibrium stresses are exactly the elements of the left kernel of the $L$-rigidity matrix.


\subsection{Stress matrices}\label{subsec:stressmatrix}

In \cite{Connelly1982},
Connelly established a sufficient condition for the global rigidity of  frameworks  in Euclidean space in terms of Laplacian matrices of undirected graphs weighted by equilibrium stresses, called \emph{stress matrices}.
Using the matrices $I_{\mathbb{Z}^d}(G)$ and $I(G)$ defined in \Cref{subsec:incidence}, one can naturally extend his idea as follows.

Given a  $\mathbb{Z}^d$-gain graph $G=(V,E)$ with edge weight $\omega:E \rightarrow \mathbb{R}$, 
the \emph{weighted Laplacian} of $G$ and the \emph{weighted $\mathbb{Z}^d$-Laplacian} of $G$ are the respective matrices
\begin{align*}
    \mathcal{L}(G,\omega):=I(G)^{\top} {\rm diag}(\omega) I(G), \qquad \mathcal{L}_{\mathbb{Z}^d}(G,\omega):=I_{\mathbb{Z}^d}(G)^{\top} {\rm diag}(\omega) I_{\mathbb{Z}^d}(G),
\end{align*}
where ${\rm diag}(\omega)$ denotes the diagonal $|E| \times |E|$ matrix whose diagonal vector is equal to $\omega$.
Alternatively,
\begin{equation*}
    \mathcal{L}_{\mathbb{Z}^d}(G,\omega)
    =
    \begin{bmatrix}
        \mathcal{L}(G,\omega) & I(G)^\top {\rm diag}(\omega) M(G)^\top \\
        M(G) {\rm diag}(\omega) I(G) & M(G) {\rm diag}(\omega) M(G)^\top
    \end{bmatrix},
\end{equation*}
where $M(G)$ is the $d \times |E|$ matrix where the column corresponding to the edge $e=(u,v,\gamma)$ is given by the vector $\gamma$.
By construction, $\ker \mathcal{L}_{\mathbb{Z}^d}(G,\omega) \supset \ker I_{\mathbb{Z}^d}(G)$ and $\ker \mathcal{L}(G,\omega) \supset \ker I(G)$, and hence $\dim \ker \mathcal{L}_{\mathbb{Z}^d}(G,\omega) \geq 1$ and $\dim \ker \mathcal{L}(G,\omega) \geq 1$.
It should be noted that any two arbitrary edge orientations chosen when constructing $I_{\mathbb{Z}^d}(G)$ (i.e., $(u,v, \gamma)$ vs $(v,u,-\gamma)$) will give the same weighted Lapacian and weighted $\mathbb{Z}^d$-Laplacian.

Now we give an explicit link between the kernel of these Laplacian matrices and the various equilibrium conditions. 
For this purpose it is convenient to use the following matrix representations of realizations:
given a realization $(p, L)\in \mathcal{R}(G)$,
the \emph{coordinate matrix} of $p$ is 
\begin{align*}
    P:=
    \begin{bmatrix}
        & \cdots & p(v) & \cdots &
    \end{bmatrix}
    \in \mathbb{R}^{d\times |V|}.
\end{align*}
We will frequently use this notation $P$ together with $L$, like $[P ~ L]\in \mathbb{R}^{d\times (|V|+d)}$, and
we will refer to it as the \emph{matrix representation} of $(p,L)$.

\begin{proposition}\label{prop:suff1}
    Let $(G,p, L)$ be a $\mathbb{Z}^d$-framework with an edge weighting $\omega :E \rightarrow \mathbb{R}$.
    \begin{enumerate}
        \item \label{prop:suff1item1} $\omega$ is an equilibrium stress of $(G,p,L)$ if and only if $[P ~ L] \mathcal{L}_{\mathbb{Z}^d}(G,\omega) = \mathbf{0}_{d \times (|V|+d)}$.
        \item \label{prop:suff1item2} $\omega$ is a fixed-lattice equilibrium stress of $(G,p,L)$ if and only if 
        \begin{align*}
            P \mathcal{L}(G,\omega) + L M(G) {\rm diag}(\omega) I(G) = \mathbf{0}_{d \times |V|}.
        \end{align*}
    \end{enumerate}
\end{proposition}

\begin{proof}
    \ref{prop:suff1item1}:
    For each edge $e = (u,v,\gamma) \in E$, define $x_e$ to be the row of $I_{\mathbb{Z}^d}(G)$ that corresponds to the edge $e$.
    Then $\mathcal{L}_{\mathbb{Z}^d}(G,\omega) =\sum_{e\in E} \omega(e) x_e^{\top} x_e$.
    Observe that, for each edge $e=(u,v,\gamma)$,  
    \begin{align*}
        [P ~ L] x_e^{\top} =p(v) + L\gamma - p(u) := \nu(e),
    \end{align*}
    and hence 
    \begin{align*}
        [P\ L] x_e^{\top} x_e = \nu(e) x_e = \left[ 0 \quad \dots \quad 0 \quad -\nu(e) \quad 0 \quad \dots \quad 0 \quad \nu(e) \quad 0 \quad \dots \quad 0 \quad \nu(e)\gamma^{\top} \right].
    \end{align*}
    Thus the column of $[P\ L] \mathcal{L}_{\mathbb{Z}^d}(G,\omega)$ corresponding to $x \in V$ is equal to the left side of \cref{eq:1},
    and the final $d\times d$ block of $[P\ L] \mathcal{L}_{\mathbb{Z}^d}(G,\omega)$ is equal to the left side of \cref{eq:2}. 
    Hence the statement follows.
    
    \ref{prop:suff1item2}:
    If we delete the $d$ rightmost columns of the matrix $[P\ L] \mathcal{L}_{\mathbb{Z}^d}(G,\omega)$,
    we obtain the matrix $P \mathcal{L}(G,\omega) + L M(G) {\rm diag}(\omega) I(G)$.
    The statement now follows from the previous part of the proof.
\end{proof}

If a $\mathbb{Z}^d$-framework $(G,p,L)$ is affinely spanning then the row vector $\hat{\mathbf{1}}^\top$ (see \Cref{subsec:incidence}) is not spanned by the row vectors of $[P\ L]$.
Together with \Cref{prop:suff1} this implies that for any affinely spanning $\mathbb{Z}^d$-framework $(G,p,L)$ with equilibrium stress $\omega$, we have $\dim \ker \mathcal{L}_{\mathbb{Z}^d}(G,\omega) \geq d+1$.
If this inequality is an equality,
then $(G,p,L)$ has some special properties.

\begin{proposition}\label{prop:afftrans}
    Let $(G,p, L)$ be an affinely spanning $\mathbb{Z}^d$-framework with an equilibrium stress $\omega :E \rightarrow \mathbb{R}$.
    Suppose that $\omega$ is also an equilibrium stress of a $\mathbb{Z}^d$-framework $(G,p',L')$.
    If the nullity of $\mathcal{L}_{\mathbb{Z}^d}(G, \omega)$ is $d+1$,
    then $(G,p',L')$ is an affine transformation of $(G,p,L)$.
\end{proposition}

\begin{proof}
    Given $[P ~ L]$ and $[P' ~ L']$ are the matrix representations of $(p,L)$ and $(p',L')$ respectively,
    it follows from \Cref{prop:suff1}\ref{prop:suff1item1} that 
    \begin{align*}
        [P\ L] \mathcal{L}_{\mathbb{Z}^d}(G,\omega) = \mathbf{0}_{d \times (|V|+d)} = [P'\ L'] \mathcal{L}_{\mathbb{Z}^d}(G,\omega).
    \end{align*}
    Since $(p,L)$ is affinely spanning,
    the row vectors of $[P ~ L]$ plus the row vector $\hat{\mathbf{1}}^\top$ (see \Cref{subsec:incidence}) form the basis of a $(d+1)$-dimensional linear space $X \subset \mathbb{R}^{1 \times (|V|+d)}$ contained in the left kernel of $\mathcal{L}_{\mathbb{Z}^d}(G,\omega)$.
    As $\mathcal{L}_{\mathbb{Z}^d}(G, \omega)$ is symmetric with $\dim \ker \mathcal{L}_{\mathbb{Z}^d}(G, \omega) = d+1$,
    it follows that $X$ is exactly the left kernel of $\mathcal{L}_{\mathbb{Z}^d}(G,\omega)$.
    It now follows that the rows of $[P'\ L']$ are contained in $X$,
    which is equivalent to $(G,p',L')$ being an affine transformation of $(G,p,L)$.
\end{proof}

A stronger statement holds when we fix our choice of lattice.

\begin{proposition}\label{prop:fltransl}
    Let $(G,p, L)$ be a $\mathbb{Z}^d$-framework with a fixed-lattice equilibrium stress $\omega :E \rightarrow \mathbb{R}$.
    Suppose that $\omega$ is also a fixed-lattice equilibrium stress of a $\mathbb{Z}^d$-framework $(G,p',L)$.
    If the nullity of $\mathcal{L}(G, \omega)$ is 1,
    then $(G,p',L)$ is a translated copy of $(G,p,L)$.
\end{proposition}

\begin{proof}
    Given $P$ and $P'$ are the coordinate matrices of $p$ and $p'$ respectively,
    it follows from \Cref{prop:suff1}\ref{prop:suff1item2} that 
    \begin{equation*}
        (P - P') \mathcal{L}(G,\omega) = \mathbf{0}_{d \times |V|} .
    \end{equation*}
    Since $\mathcal{L}_{\mathbb{Z}^d}(G, \omega)$ is symmetric, $\dim \ker \mathcal{L}(G,\omega) = 1$ and the left kernel of $\mathcal{L}(G,\omega)$ contains the all one row vector $\mathbf{1} = [1 ~ \cdots ~ 1]$, 
    we have $P' = P + z \mathbf{1}$ for some $z \in \mathbb{R}^d$.
    Hence $(G,p',L)$ a translated copy of $(G,p,L)$.
\end{proof}

We make the following observation for the special case where a $\mathbb{Z}^d$-framework $(G,p,L)$ is flat (i.e., $\rank L \leq d-1$) but affinely spanning.
If we choose a fixed-lattice equilibrium stress $\omega$ of $(G,p,L)$,
then we must have $\dim \ker \mathcal{L}(G,\omega) \geq 2$.
This stems from the following observation.
Choose $x \in \ker L \setminus \{\mathbf{0}\}$ and define the reflection $T$ where $x \mapsto -x$ and $y \mapsto y$ if $y \perp x$,
which implies $TL = L$.
If we now fix $p' = T \circ p$ then \Cref{prop:suff1}\ref{prop:suff1item2} implies that the $\mathbb{Z}^d$-framework $(G,p', L)$ will have the same fixed-lattice equilibrium stresses as $(G,p,L)$,
however $(G,p',L)$ is not a translated copy of $(G,p,L)$ since the latter is affinely spanning.

\subsection{Converting fixed-lattice equilibrium stresses into equilibrium stresses}

An obvious method for considering fixed-lattice rigidity using the flexible-lattice model is to add enough loops to our $\mathbb{Z}^d$-framework so as to ``lock'' the lattice into a single possibility.
In this subsection we describe a method of doing this so that it provides a bijection between the fixed-lattice equilibrium stresses of the original $\mathbb{Z}^d$-framework and the equilibrium stresses of the $\mathbb{Z}^d$-framework with added loops.

We begin with the following lemma.
For the following we use the shorthand $A^{-\top} = (A^{-1})^\top$.

\begin{lemma}\label{lem:sub}
    If $\omega$ is a fixed-lattice equilibrium stress of a non-flat $\mathbb{Z}^d$-framework $(G,p,L)$, then
    \begin{equation*}
        \mathcal{L}_{\mathbb{Z}^d}(G,\omega)
        =
        \begin{bmatrix}
            \mathcal{L}(G,\omega) & -\mathcal{L}(G,\omega)P^\top L^{-\top} \\
            -L^{-1}P \mathcal{L}(G,\omega) & M(G) {\rm diag}(\omega) M(G)^\top
        \end{bmatrix}.
    \end{equation*}
\end{lemma}

\begin{proof}
    Apply the substitution $M(G) {\rm diag}(\omega) I(G) = -L^{-1} P \mathcal{L}(G,\omega)$ (a consequence of \Cref{prop:suff1}\ref{prop:suff1item2}) to $\mathcal{L}_{\mathbb{Z}^d}(G,\omega)$.
\end{proof}

\begin{lemma}\label{lem:extend}
    Let $\omega$ be a fixed-lattice equilibrium stress of the non-flat $\mathbb{Z}^d$-framework $(G,p,L)$,
    and let $F$ be a set of $\binom{d+1}{2}$ loops at an arbitrary vertex $u$ so that the set
    \begin{equation*}
        \Gamma := \left\{ \gamma \gamma^\top : (u,u,\gamma) \in F\right\}
    \end{equation*}
    is a basis of the space of $d \times d$ symmetric matrices.
    Then there exists $\mu \in \mathbb{R}^F$ so that $(\omega,\mu)$ is an equilibrium stress of $(G+F,p,L)$.
\end{lemma}

\begin{proof}
    By abuse of notation, we consider $M(F) = M(H)$ where $H = (V,F)$.
    For any $\mu \in \mathbb{R}^E$, we have by \Cref{lem:sub} that 
    \begin{equation*}
        \mathcal{L}_{\mathbb{Z}^d}(G+F,\omega)
        =
        \begin{bmatrix}
            \mathcal{L}(G,\omega) & -\mathcal{L}(G,\omega)P^\top L^{-\top} \\
            -L^{-1}P \mathcal{L}(G,\omega) & M(G) {\rm diag}(\omega) M(G)^\top + M(F) {\rm diag}( \mu) M(F)^\top
        \end{bmatrix} \\
    \end{equation*}
    By \Cref{prop:suff1},
    we have that $(\omega,\mu)$ is an equilibrium stress of $(G+F,p,L)$ if and only if 
    \begin{gather*}
        -P \mathcal{L}(G,\omega) P^\top L^{-\top} + LM(G) {\rm diag}(\omega) M(G)^\top = -L M(F) {\rm diag}( \mu) M(F)^\top \\
        \Updownarrow \\
        \sum_{f=(u,u,\gamma) \in F} \mu(f) L \gamma \gamma^\top L^\top  = P \mathcal{L}(G,\omega) P^\top - LM(G) {\rm diag}(\omega) M(G)^\top L^\top.
    \end{gather*}
    Since $\Gamma$ spans the space of $d \times d$ symmetric matrices, we now choose $\mu$ so that the above equality holds.
\end{proof}

\subsection{Tensegrities}

One can extend all the above definitions to tensegrities.
Roughly speaking, a tensegrity is obtained from a framework by replacing
some (or all) of the equalities for stiff bars by inequalities for cables and struts, where a cable allows the distance between its endpoints to shrink (but not expand) and a strut allows the distance between its endpoints to expand (but not shrink), see e.g.  \cite{rothwhi,Connelly1982,conguest}.
With this we now give the formal definition for the periodic analog to a tensegrity.

\begin{definition}
    A \emph{$\mathbb{Z}^d$-tensegrity} is a $\mathbb{Z}^d$-framework $(G, p, L)$, where the edge set $E$ is partitioned into the sets $E_0$, $E_+$ and $E_-$ representing \emph{bars}, \emph{cables} and \emph{struts}, respectively. 
\end{definition}

A $\mathbb{Z}^d$-tensegrity $(G,p,L)$ \emph{dominates} a different $\mathbb{Z}^d$-tensegrity $(G,p', L')$ if 
\begin{align*}
    \|p(v) + L \gamma -p(u)\| &= \|p'(v) + L' \gamma -p'(u)\| \qquad \textrm{ for all } e = (u,v,\gamma) \in E_0\\
    \|p(v) + L \gamma -p(u)\| &\geq \|p'(v) + L' \gamma -p'(u)\| \qquad \textrm{ for all } e = (u,v,\gamma) \in E_+\\
    \|p(v) + L \gamma -p(u)\| &\leq \|p'(v) + L' \gamma -p'(u)\| \qquad \textrm{ for all } e = (u,v,\gamma) \in E_-.
\end{align*}
We now define a $\mathbb{Z}^d$-tensegrity $(G,p,L)$ to be \emph{globally rigid} if every $\mathbb{Z}^d$-tensegrity dominated by $(G,p,L)$ is congruent to $(G,p,L)$.
Similarly, $(G,p,L)$ is \emph{fixed-lattice globally rigid} if every $\mathbb{Z}^d$-tensegrity with lattice $L$ that is dominated by $(G,p,L)$ is congruent to $(G,p,L)$.

An edge weighting $\omega:E \rightarrow \mathbb{R}$ of a $\mathbb{Z}^d$-tensegrity $(G,p,L)$ is {\em proper} if $\omega(e) \geq 0$ for all $e \in E_+$ and $\omega(e) \leq 0$ for all $e \in E_-$.
An edge weighting is now said to be (i) a \emph{fixed-lattice equilibrium stress} of $(G,p,L)$ if it is proper and satisfies \cref{eq:1},
and (ii) an \emph{equilibrium stress} of $(G,p,L)$ if it is proper and satisfies both \cref{eq:1} and \cref{eq:2}.

\section{Sufficient global rigidity conditions for tensegrities}\label{sec:suff}

\subsection{A sufficient global rigidity condition for tensegrities}

We begin the subsection by introducing the following concept.
Recall the vector representation $[p^\top ~ \ell^\top]^\top$ for each pair $(p,L) \in \mathcal{R}(G)$ as described in \Cref{subsec:coord}.
For a $\mathbb{Z}^d$-gain graph $G$ with edge weight $\omega:E \rightarrow \mathbb{R}$, 
we define its \emph{energy function} ${\cal E}_{G,\omega} : {\cal R}(G) \rightarrow \mathbb{R}$ by 
\begin{equation}\label{eq:4}
    {\cal E}_{G,\omega}(p,L):= \frac{1}{2}
    \begin{bmatrix} 
        p^{\top} & \ell^{\top} 
    \end{bmatrix} 
    (\mathcal{L}_{\mathbb{Z}^d}(G,\omega) \otimes I_d) 
    \begin{bmatrix}
        p \\ \ell 
    \end{bmatrix}
\end{equation}
(here $\otimes$ represents the Kronecker product).
It is immediate that each energy function is a quadratic form.
As such,
an energy function is convex if and only if the weighted $\mathbb{Z}^d$-Laplacian $\mathcal{L}_{\mathbb{Z}^d}(G,\omega)$ is positive semidefinite.

We can rewrite \cref{eq:4} in terms of the matrix representation of a pair $(p,L)$ (see \Cref{subsec:stressmatrix}):
\begin{equation}\label{eq:4.1}
    {\cal E}_{G,\omega}(p,L) = \frac{1}{2}{\rm Tr}\left([P\ L] \mathcal{L}_{\mathbb{Z}^d}(G,\omega) [P\ L]^{\top}\right).
\end{equation}
Given $x_e$ is the row vector of $I_{\mathbb{Z}^d}(G)$ corresponding to the edge $e = (u,v, \gamma)$, 
we recall from \Cref{prop:suff1} that $[P ~ L] x_e^\top =\nu(e)$, 
where $\nu(e) := p(v) + L\gamma - p(u)$.
Since $\mathcal{L}_{\mathbb{Z}^d}(G,\omega) = \sum_{e \in E} \omega(e) x_e^\top x_e$,
\cref{eq:4.1} shows that
\begin{align*}
    {\cal E}_{G,\omega}(p,L)&= \frac{1}{2}{\rm Tr}\left([P\ L] \mathcal{L}_{\mathbb{Z}^d}(G,\omega) [P\ L]^{\top}\right)= \frac{1}{2}\sum_{e\in E}\omega(e){\rm Tr}\left([P\ L]x_e^{\top} x_e [P\ L]^{\top}\right) \\
    &=\frac{1}{2}\sum_{e\in E}\omega(e) {\rm Tr}\left(\nu(e)\nu(e)^{\top}\right)=\frac{1}{2}\sum_{e\in E}\omega(e) \|\nu(e)\|^2.
\end{align*}
Hence another equivalent representation of the energy function for a pair $(p,L) \in \mathcal{R}(G)$ is the equation
\begin{equation}\label{eq:4.2}
    {\cal E}_{G,\omega}(p,L)=\frac{1}{2}\sum_{e=(u,v,\gamma) \in E} \omega(e) \|p(v) + L\gamma - p(u)\|^2 = \frac{1}{2}\omega^\top f_{G}(p,L).
\end{equation}
This last representation allows us to easily see that the gradient of $\mathcal{E}_{G,\omega}$ at $(p,L) \in \mathbb{R}(G)$ is the vector
\begin{equation}\label{eq:4.3}
     \nabla \mathcal{E}_{G,\omega}(p,L) = (\mathcal{L}_{\mathbb{Z}^d}(G,\omega) \otimes I_d) 
    \begin{bmatrix}
        p \\ \ell 
    \end{bmatrix}
    = \omega^\top R(G,p,L)
\end{equation}
Hence an edge weighting $\omega :E \rightarrow \mathbb{R}$ is an equilibrium stress of a $\mathbb{Z}^d$-framework $(G,p,L)$ if and only if $\nabla \mathcal{E}_{G,\omega}(p,L) = 0$.

With the concept of energy functions in place,
we are almost ready to prove a sufficient global rigidity condition.
We first need the following important definition.

\begin{definition}\label{def:conic}
    For a $\mathbb{Z}^d$-tensegrity in $\mathbb{R}^d$,
    we say that the edge directions of $(G,p,L)$ \emph{lie on a conic at infinity} if there exists a non-zero symmetric $d\times d$ matrix $Q$ satisfying
    \begin{align*}
        (p(v) +L\gamma - p(u))^\top Q (p(v) +L\gamma - p(u)) = 0
    \end{align*}
    for all $e=(u,v,\gamma) \in E$.    
\end{definition}

\Cref{def:conic} can also be characterised via equivalent affine transformations.

\begin{lemma}\label{l:conic equiv}
    Let $(G,p,L)$ be a $\mathbb{Z}^d$-framework.
    Then there exists a $d \times d$ matrix $A$ with $A^\top A - I_d$ such that $(G,p,L)$ is equivalent to $(G,p',L')$ when $p'(v) = Ap(v)$ for all $v \in V$ and $L' = AL$ if and only if the edge directions of $(G,p,L)$ do not lie on a conic at infinity with respect to the non-zero symmetric $d \times d$ matrix $A^\top A -I_d$.
\end{lemma}

\begin{proof}
    Choose any matrix $A$ and construct $(G,p',L')$ from $A$.
    Then $(G,p,L)$ and $(G,p',L')$ are equivalent if and only if for each edge $e=(u,v,\gamma)$ we have
    \begin{align*}
        (p'(v) + L'\gamma - p'(u))^\top (p'(v) + L'\gamma - p'(u)) &= (p(v) + L\gamma - p(u))^\top A^\top A (p(v) + L\gamma - p(u))\\
        &= (p(v) + L\gamma - p(u))^\top (p(v) + L\gamma - p(u)).
    \end{align*}
    By rearranging these equations,
    we see that
    \begin{align*}
        (p(v) + L\gamma - p(u))^\top (A^\top A -I_d) (p(v) + L\gamma - p(u))= 0
    \end{align*}
    for every edge $e=(u,v,\gamma)$.
\end{proof}

\begin{lemma}\label{lem:affinerigid}
    Let $(G,p,L)$ and $(G,p',L')$ be two $\mathbb{Z}^d$-frameworks.
    If
    \begin{itemize}
        \item $(G,p,L)$ and $(G,p',L')$ are equivalent,
        \item $(G,p',L')$ is an affine transformation of $(G,p,L)$, and
        \item the edge directions of $(G,p,L)$ do not lie on a conic,
    \end{itemize}
    then $(G,p,L)$ and $(G,p',L')$ are congruent.
\end{lemma}

\begin{proof}
    Fix $d \times d$ matrix $M$ and a vector $z \in \mathbb{R}^d$ such that $p'(v) = Mp(v) +z$ for all $v \in V$ and $L' = ML$.
    For each $e=(u,v,\gamma)$,
    define the vectors $\nu(e) := (p(v) +L\gamma - p(u))$ and $\nu'(e) := (p'(v) +L'\gamma - p'(u)) = M\nu(e)$.
    We now note that for each $e=(u,v,\gamma)$,
    \begin{align*}
        \nu(e)^\top (M^\top M - I_d) \nu(e) = \nu(e)^\top M^\top M \nu(e) - \nu(e)^\top \nu(e)  =
        \|\nu'(e)\|^2 - \|\nu(e)\|^2,
    \end{align*}
    with the last equality following from $(G,p,L)$ and $(G,p',L')$ being equivalent.
    Since the edge directions of $(G,p,L)$ do not lie on a conic at infinity, it follows that $M^{\top}M=I_d$, that is, $M$ is orthogonal and $(G,p,L)$ is congruent to $(G,p',L')$.
\end{proof}

With this final piece in place,
we are now ready to prove the  main result of the section.

\begin{theorem}\label{thm:suffpsd}
    An affinely spanning $\mathbb{Z}^d$-tensegrity $(G,p,L)$ is globally rigid if it has an equilibrium stress $\omega:E \rightarrow \mathbb{R}$ such that
    \begin{itemize}
        \item $\dim \ker \mathcal{L}_{\mathbb{Z}^d}(G,\omega) = d+1$ and 
        \item $\mathcal{L}_{\mathbb{Z}^d}(G,\omega)$ is positive semidefinite,
    \end{itemize}
    and the edge directions of $(G,p,L)$ do not lie on a conic at infinity.
\end{theorem}
    
\begin{proof}
    Consider the energy function ${\cal E}_{G,\omega}$.
    Since $\mathcal{L}_{\mathbb{Z}^d}(G,\omega)$ is positive semidefinite, 
    the quadratic form  ${\cal E}_{G,\omega}$ is convex.
    Hence it follows from \cref{eq:4.2} that $(p',L') \in \mathcal{R}(G)$ is a minimizer of ${\cal E}_{G,\omega}$ if and only if $\omega$ is an equilibrium stress of $(G,p',L')$.
    
    Fix a $\mathbb{Z}^d$-tensegrity $(G,p',L')$ that is dominated by $(G,p,L)$.
    By the definition of equilibrium stresses of $\mathbb{Z}^d$-tensegrities we have
    \begin{equation}\label{eq:stressineq}
        \omega(e)\| p'(v) + L' \gamma - p' (v) \|^2 \leq \omega(e)\| p(v) + L \gamma - p (v) \|^2
    \end{equation}
    for each edge $e = (u,v,\gamma) \in E$.
    Hence from \cref{eq:4.2} we have ${\cal E}_{G,\omega}(p',L') \leq {\cal E}_{G,\omega}(p,L)$.
    Since $(p,L)$ is a minimizer of ${\cal E}_{G,\omega}$,
    it follows that ${\cal E}_{G,\omega}(p',L') = {\cal E}_{G,\omega}(p,L)$.
    This implies two things:
    (i) $(G,p',L')$ is equivalent to $(G,p,L)$ since \cref{eq:stressineq} must be an equality for each $e \in E$,
    and (ii) $(p',L')$ is a minimizer of ${\cal E}_{G,\omega}$.
    As shown above, the latter property implies $\omega$ is an equilibrium stress of $(G,p',L')$.
    By \Cref{prop:afftrans},
    $(G,p',L')$ is an affine transformation of $(G,p,L)$.
    Hence $(G,p,L)$ and $(G,p',L')$ are congruent by \Cref{lem:affinerigid}.
\end{proof}

\begin{remark}
 A finite tensegrity is called \emph{super stable} if it satisfies  conditions analogous to those in \Cref{thm:suffpsd} \cite[Chapter 63]{HandDCG}.
  Super stable tensegrities are in particular \emph{universally rigid} \cite{Connelly1982}, i.e., any other tensegrity on the given graph, with the same partition of edges into cables, struts and bars, \emph{in any dimension}, that is dominated by the given tensegrity, is congruent to it. Similarly, we may call a $\mathbb{Z}^d$-tensegrity which satisfies the conditions of \Cref{thm:suffpsd}  {\em super stable}.
\end{remark}

\begin{example}\label{ex:flex1}
    For a single-vertex $\mathbb{Z}^d$-framework, global rigidity is achieved once the lattice is fully constrained, i.e., after adding $\binom{d+1}{2}$ independent loops.  Any additional loop then creates an equilibrium stress without imposing further constraints. Geometrically, this means the equilibrium stress is entirely determined by the gain assignments of the loops and is independent of the vertex position. Algebraically, this is reflected in the fact that the associated stress matrix $\mathcal{L}_{\mathbb{Z}^d}(G,\omega)$ is the $(d+1)\times (d+1)$ zero matrix, and this matrix is trivially positive semidefinite. Its kernel has dimension $d+1$, and since there are $\binom{d+1}{2}$ independent loops, the edge directions do not lie on a conic at infinity. Hence,  \Cref{thm:suffpsd} certifies global rigidity. We note that global rigidity and infinitesimal rigidity are the same for this special example; see \cite[Theorem 3.12]{BS10}.
        \end{example}

\begin{example}\label{ex:flex2}
     Consider the \(\mathbb Z^2\)-framework  shown in \Cref{fig:two-vertex-framework}
    consisting of two vertices $v_1$ and $v_2$ placed at $p(v_1)=(0,0)$ and $p_2=(0.5,0)$, with lattice $L=I_2$ and edges $e_1=(v_1,v_2,(0,0))$, $e_2=(v_1,v_2,(-1,0))$, $e_3=(v_1,v_1,(0,1))$, $e_4=(v_1,v_1,(1,1))$, and $e_5=(v_1,v_1,(-1,1))$. Define the edge weight \(\omega:E\to\mathbb R\) by
\[
\omega(e_1)=4,\qquad \omega(e_2)=4,\qquad \omega(e_3)=2,\qquad \omega(e_4)=-1,\qquad \omega(e_5)=-1.
\] It is easy to check that $\omega$ is an equilibrium stress since both \cref{eq:1} and \cref{eq:2} hold.

Now, we have 
\[
I_{\mathbb Z^2}(G)=
\begin{pmatrix}
-1 & 1 & 0 & 0 \\
 -1 &  1 & -1 & 0 \\
 0 & 0 & 0 & 1 \\
 0 &  0 & 1 & 1\\
 0 &  0 & -1 & 1
\end{pmatrix} \quad
\textrm{ and }\quad 
\mathcal{L}_{\mathbb Z^2}(G,\omega)=
\begin{pmatrix}
8 & -8 &  4 & 0\\
-8 &  8 & -4 & 0\\
 4 & -4 &  2 & 0\\
 0 &  0 &  0 & 0
\end{pmatrix}.
\]
Since  $\mathcal{L}_{\mathbb Z^2}(G,\omega)$ satisfies the conditions of  \Cref{thm:suffpsd} and the edge directions do not lie on a conic of infinity, we may conclude that the $\mathbb{Z}^2$-framework is globally rigid. (If the bars corresponding to $e_1,e_2$ and $e_3$ are replaced by cables and the remaining two bars are replaced by struts, then the resulting tensegrity is also globally rigid.) Note that perturbing $p(v_2)$ so that it is no longer collinear with its neighbors in the covering framework yields a $\mathbb{Z}^2$-framework that is not globally rigid. This is because the vertex  $v_2$ has degree two, allowing  it to be locally reflected in the line defined by its two neighbors. This provides an equivalent but non-congruent realization.
\end{example}

\begin{figure}[htp]
\begin{center}
\begin{tikzpicture}[scale=1.5,
    vertex/.style={circle,fill=black,inner sep=1.2pt},
    edge/.style={line width=0.6pt},
    dedge/.style={line width=1pt, line cap=round}
]

\node[vertex] (p1) at (0,0) {};
\node[vertex] (p2) at (1,0) {};

\draw[->, >=stealth, bend left=20] (p1) to (p2);   
\draw[->, >=stealth, bend right=20] (p1) to (p2);  

\draw[->, >=stealth] (p1) .. controls (-1,0.5) and (-1,-0.5) .. (p1); 
\draw[->, >=stealth] (p1) .. controls (-0.5,-1) and (0.5,-1) .. (p1);  
\draw[->, >=stealth] (p1) .. controls (-0.5,1) and (0.5,1) .. (p1);    

      \node [draw=white, fill=white] (a) at (-1.1,0) {\tiny{$(1,1)$}}; 
 \node [draw=white, fill=white] (a) at (0,-1) {\tiny{$(-1,1)$}}; 
  \node [draw=white, fill=white] (a) at (0,1) {\tiny{$(0,1)$}}; 
   \node [draw=white, fill=white] (a) at (0.6,0.3) {\tiny{$(0,0)$}}; 
    \node [draw=white, fill=white] (a) at (0.6,-0.3) {\tiny{$(-1,0)$}}; 

\begin{scope}[xshift=4cm, yshift=0cm, scale=0.5]

\coordinate (a1) at (1,0);
\coordinate (a2) at (0,1);

\coordinate (w1) at (0,0);
\coordinate (w2) at (0.5,0);

\foreach \i in {-2,-1,0,1,2}{
  \foreach \j in {-2,-1,0,1,2}{
    \pgfmathsetmacro{\dx}{\i*1}
    \pgfmathsetmacro{\dy}{\j*1}

    \path (w1)++(\dx,\dy) coordinate (p1);
    \path (w2)++(\dx,\dy) coordinate (p2);

    \draw[line width=0.5pt] (p1)--(p2); 
    \draw[line width=0.5pt] (p1)--($(p2)+(-1,0)$); 
    \draw[line width=0.5pt] (p1)--($(p1)+(0,1)$); 
    \draw[line width=0.5pt] (p1)--($(p1)+(1,1)$); 
    \draw[line width=0.5pt] (p1)--($(p1)+(-1,1)$); 

    \node[vertex] at (p1) {};
    \node[vertex] at (p2) {};
  }
}
\end{scope}

\end{tikzpicture}
\end{center}
\caption{Shown on the left is the globally rigid $\mathbb{Z}^2$-framework described in \Cref{ex:flex2}. The corresponding covering framework is shown on the right.} 
\label{fig:two-vertex-framework}
\end{figure}
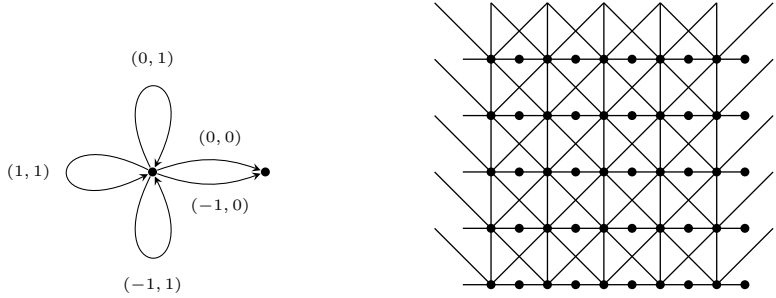

In \Cref{appb} we describe  a general procedure to construct globally rigid (in fact, super stable) $\mathbb{Z}^d$-tensegrities from finite super stable tensegrities.

\subsection{A sufficient fixed-lattice global rigidity condition for tensegrities}

It is relatively straightforward to adapt the proof of \Cref{thm:suffpsd} to obtain an analogous result for the case of fixed-lattice global rigidity.
For a $\mathbb{Z}^d$-gain graph $G$ with edge weight $\omega: E \rightarrow \mathbb{R}$ and lattice $L$,
we define the \emph{$L$-energy function} by
\begin{align*}
    \mathcal{E}_{G,\omega,L} : (\mathbb{R}^d)^V \rightarrow \mathbb{R}, ~ p' \mapsto \mathcal{E}_{G,\omega} (p',L).
\end{align*}
We now make two important observations about $\mathcal{E}_{G,\omega,L}$.
Firstly,
it follows from \cref{eq:4.2} that
\begin{equation}\label{eq:fl1}
    \mathcal{E}_{G,\omega,L}(p) = \mathcal{E}_{G,\omega}(p,L) = \frac{1}{2}\omega^\top f_{G}(p) = \frac{1}{2}\omega^\top f_{G,L}(p),
\end{equation}
hence
\begin{equation}\label{eq:fl2}
    \nabla \mathcal{E}_{G,\omega,L}(p)  = \omega^\top R_L(G,p).
\end{equation}
From the observation $R_L(G,p)p = \frac{1}{2}f_{G,L}(p)$,
it follows that $\mathcal{E}_{G,\omega,L}(p) = 0$ and $\nabla \mathcal{E}_{G,\omega,L}(p) = 0$ if $\omega$ is a fixed-lattice equilibrium stress of $(G,p,L)$.
Secondly,
there exist $b \in \mathbb{R}^{d|V|}$ and $c \in \mathbb{R}$ such that for any $p \in \mathbb{R}^{d|V|}$ we have
\begin{equation}\label{eq:energyfixedlattice}
    {\cal E}_{G,\omega,L}(p) = p^\top (\mathcal{L}(G,\omega) \otimes I_d) p + b^\top p + c.
\end{equation}
Hence ${\cal E}_{G,\omega,L}$ is convex if and only if the weighted Laplacian $\mathcal{L}(G,\omega)$ is positive semidefinite.
By combining our two observations we note that, if $\mathcal{L}(G,\omega)$ is positive semidefinite,
then $p$ is a minimizer of ${\cal E}_{G,\omega,L}$ if and only if $\omega$ is a fixed-lattice equilibrium stress of $(G,p,L)$.

\begin{theorem}\label{thm:fixedpsd}
    A $\mathbb{Z}^d$-tensegrity $(G,p,L)$ is fixed-lattice globally rigid if it has a fixed-lattice equilibrium stress $\omega:E \rightarrow \mathbb{R}$ such that
    \begin{itemize}
        \item $\dim \ker \mathcal{L}(G,\omega) = 1$ and 
        \item $\mathcal{L}(G,\omega)$ is positive semidefinite.
    \end{itemize}
\end{theorem}

\begin{proof}
    Consider the fixed-lattice energy function ${\cal E}_{G,\omega,L}$.
    Since $\mathcal{L}(G,\omega)$ is positive semidefinite,
    ${\cal E}_{G,\omega,L}$ is convex.
    Hence $p' \in (\mathbb{R}^d)^V$ is a minimizer of ${\cal E}_{G,\omega,L}$ if and only if $\omega$ is a fixed-lattice equilibrium stress of $(G,p',L)$.
    
    Fix a $\mathbb{Z}^d$-tensegrity $(G,p',L)$ that is dominated by $(G,p,L)$.
    By the definition of fixed-lattice equilibrium stresses of $\mathbb{Z}^d$-tensegrities we have
    \begin{equation}\label{eq:stressineq2}
        \omega(e)\| p'(v) + L' \gamma - p' (v) \|^2 \leq \omega(e)\| p(v) + L \gamma - p (v) \|^2
    \end{equation}
    for each edge $e = (u,v,\gamma) \in E$.
    Hence from \cref{eq:4.2} we have ${\cal E}_{G,\omega,L}(p') \leq {\cal E}_{G,\omega}(p)$.
    Since $p$ is a minimizer of ${\cal E}_{G,\omega,L}$,
    it follows that ${\cal E}_{G,\omega,L}(p') = {\cal E}_{G,\omega,L}(p)$.
    This implies $p'$ is also a minimizer of ${\cal E}_{G,\omega,L}$,
    and so $\omega$ is also a fixed-lattice equilibrium stress of $(G,p',L)$.
    Hence $(G,p,L)$ and $(G,p',L')$ are congruent by \Cref{prop:fltransl}.
\end{proof}

\begin{example}\label{ex:hex} 
    Consider the $\mathbb{Z}^2$-framework and its covering shown in \Cref{fig:hex}. This is the crystal structure of graphene.
    Up to scalar multiplication, there is only one fixed-lattice equilibrium stress $\omega$, and this stress assigns the value 1 to every edge.
    The weighted Laplacian matrix $\mathcal{L}(G,\omega)$ is now the actual Laplacian matrix $\mathcal{L}(G)$, and hence $\mathcal{L}(G,\omega)$ is positive semidefinite with nullity 1 by the connectivity of $G$.
    Hence by \Cref{thm:fixedpsd},
    $(G,p,L)$ is fixed-lattice globally rigid.
    However, no generic realization of $G$ is globally rigid since the $\mathbb{Z}^2$-gain graph does not satisfy the combinatorial characterization described in \cite[Theorem 4.2]{kst21}.
    Observe that $(G,p,L)$ is still globally rigid if every bar is replaced by a cable; however nothing can be concluded if any cable is replaced by a strut (since then $\omega$ is no longer proper).
\end{example}

\begin{figure}[htp]
\begin{center}
    \begin{tikzpicture}[scale = 1.25,
        vertex/.style={circle,fill=black,inner sep=1.2pt},
        edge/.style={line width=0.6pt},
        dedge/.style={line width=1pt, line cap=round}
      ]

    \coordinate (A) at (-2.25, -1.5*0.866);
    \coordinate (B) at (0.75, -1.5*0.866);
    \coordinate (C) at (2.25, 1.5*0.866);
    \coordinate (D) at (-0.75, 1.5*0.866);
    \draw[thick, gray!50, fill=gray!50] (A) -- (B) -- (C) -- (D) -- cycle;

    \draw[dedge,-stealth] (A)--(B);
    \draw[dedge,-stealth] (A)--(D);
    
\node [draw=white, fill=white] (name1) at (-0.8,-1.6) {\textcolor{black}{\tiny{$(1,0)$}}};

\node [draw=white, fill=white] (name1) at (-2.5, 0) {\textcolor{black}{\tiny{$(1/2,\sqrt{3}/2)$}}};

\node [draw=gray!50, fill=gray!50] (a) at (0.5,-0.2) {\textcolor{red}{\tiny{$(1,0)$}}};

\node [draw=gray!50, fill=gray!50] (a) at (-0.5,-0.3) {\textcolor{RoyalBlue}{\tiny{$(1,1)$}}};

\node [draw=gray!50, fill=gray!50] (a) at (0,-0.7) {\textcolor{ForestGreen}{\tiny{$(-1,1)$}}};

    \node[vertex] (1) at (1,0) {};
    \node[vertex] (2) at (0.5,0.866) {};
    \node[vertex] (3) at (-0.5,0.866) {};
    \node[vertex] (4) at (-1,0) {};
    \node[vertex] (5) at (-0.5,-0.866) {};
    \node[vertex] (6) at (0.5,-0.866) {};

    \draw[edge] (1)--(2);
    \draw[edge] (2)--(3);
    \draw[edge] (3)--(4);
    \draw[edge] (4)--(5);
    \draw[edge] (5)--(6);
    \draw[edge] (6)--(1);

    \draw[dedge,red,-stealth] (1)--(4);
    \draw[dedge,RoyalBlue,-stealth] (2)--(5);
    \draw[dedge,ForestGreen,-stealth] (3)--(6);

    \hspace{1.5cm}
    \begin{scope}[xshift=4.2cm, yshift=0cm, scale=0.3]
      \coordinate (a1) at (3,0);         
      \coordinate (a2) at (1.5,2.598);   

      \coordinate (v1) at (1,0);
      \coordinate (v2) at (0.5,0.866);
      \coordinate (v3) at (-0.5,0.866);
      \coordinate (v4) at (-1,0);
      \coordinate (v5) at (-0.5,-0.866);
      \coordinate (v6) at (0.5,-0.866);

      \foreach \i in {-2,...,2}{
        \foreach \j in {-2,...,2}{
          \pgfmathsetmacro{\x}{\i*3 + \j*1.5}
          \pgfmathsetmacro{\y}{\j*2.598}

          \path (v1)++(\x,\y) coordinate (w1);
          \path (v2)++(\x,\y) coordinate (w2);
          \path (v3)++(\x,\y) coordinate (w3);
          \path (v4)++(\x,\y) coordinate (w4);
          \path (v5)++(\x,\y) coordinate (w5);
          \path (v6)++(\x,\y) coordinate (w6);

          \draw (w1)--(w2)--(w3)--(w4)--(w5)--(w6)--(w1);

          \draw[red, line width=1pt] (w1) -- ($(w4)+(a1)$);
          \draw[RoyalBlue, line width=1pt] (w2) -- ($(w5)+(a2)$);
          \draw[ForestGreen, line width=1pt] (w3) -- ($(w6)-(a1)+(a2)$);

          \foreach \p in {w1,w2,w3,w4,w5,w6}{
            \node[vertex] at (\p) {};
          }
        } 
      } 
    \end{scope}

    \end{tikzpicture}
\end{center}
\caption{Shown on the left is the $\mathbb{Z}^2$-framework described in \Cref{ex:hex} (together with the two vectors spanning the lattice), where the black edges of the hexagon have trivial gain and arbitrary orientation. The corresponding covering framework is shown on the right.}
\label{fig:hex}
\end{figure}
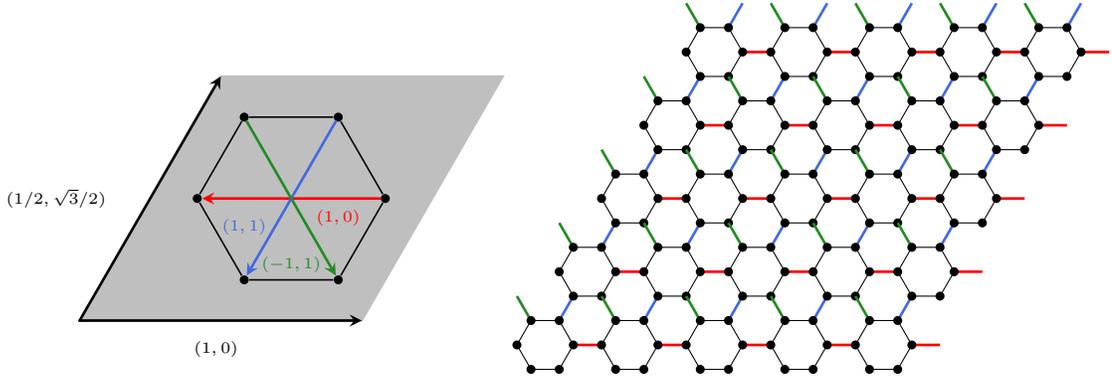

We conclude the section with a special family of tensegrities.
The $\mathbb{Z}^d$-tensegrity given in \Cref{ex:hex} is a special type of tensegrity known as a \emph{spiderweb}:
a non-flat $\mathbb{Z}^d$-tensegrity $(G,p,L)$ for which $G$ is connected with rank $d$ and each edge is a cable.
The sufficient condition for fixed-lattice global rigidity given in \Cref{thm:fixedpsd} can be simplified for this class of tensegrities.

\begin{corollary}\label{cor:spiderweb}
    Let $(G,p,L)$ be a spiderweb.
    If $(G,p,L)$ has an equilibrium stress $\omega$ with $\omega(e)>0$ for all $e\in E$, 
    then $(G,p,L)$ is fixed-lattice globally rigid.
\end{corollary}

\begin{proof}
    Since $G$ is connected, its Laplacian matrix has rank $n-1$, 
    and hence the same is true for its weighted Laplacian matrix $\mathcal{L}(G,\omega)$.
    Moreover, $\mathcal{L}(G,\omega)$ is positive semidefinite since $\omega$ is positive on each edge. Thus, \Cref{thm:fixedpsd} applies.
\end{proof}

\begin{remark}
    \Cref{cor:spiderweb} was previously observed by Delgado-Friedrichs; see \cite[Theorem 4]{DF05}.
\end{remark}

A spiderweb $(G,p,L)$ with a strictly positive equilibrium stress $\omega$ collapses to a point if the lattice $L$ is not fixed. 
To see this, note that in this case we have $\rank I_{\mathbb{Z}^d}(G) = |V| +d-1$ by \Cref{prop:1},
 and hence $\rank \mathcal{L}_{\mathbb{Z}^d}(G,\omega) = |V| + d - 1$, 
or equivalently $\dim \ker \mathcal{L}_{\mathbb{Z}^d}(G,\omega) = 1$, by the positivity of $\omega$. 
It now follows easily from \Cref{prop:suff1} and the structure of the elements in the kernel of $\mathcal{L}_{\mathbb{Z}^d}(G,\omega)$ that all points of $(G,p,L)$ are coincident and the lattice is the zero matrix.

\section{Sufficient global rigidity conditions for generic periodic frameworks}\label{sec:suffgen}

Both \Cref{thm:suffpsd} and \Cref{thm:fixedpsd} require that the Laplacian matrix in question is positive semidefinite.
This is unfortunately a rather strong assumption, and many frameworks exist that are (fixed-lattice) globally rigid with no positive semidefinite (fixed-lattice) equilibrium stresses.
In this section we prove that this assumption can be dropped if the $\mathbb{Z}^d$-framework in question is generic.
These results provide the sufficiency conditions for \Cref{mainthm:flexlattice} and \Cref{mainthm:fixedlattice}.
Our proof for these results follows a similar technique as Connelly's original proof for finite frameworks \cite{conggr}, with some adjustments to account for the periodic structure.

For this section we require the following important technical result, a proof of which can be found in \Cref{app:int}:

\begin{restatable}{lemma}{connellylemma}\label{l:connelly lemma}
	Let $f:\mathbb{R}^m \rightarrow \mathbb{R}^n$ be a polynomial map such that,
    given some finite set $S \subset \mathbb{R}$,
    every coefficient is contained in the field $\mathbb{Q}(S)$.
	Let $y \in \mathbb{R}^m$ be a point whose coordinates form an algebraically independent set of size $m$ over the field $\mathbb{Q}(S)$.
    Then for any $x \in \mathbb{R}^m$ with $f(x) = f(y)$,
    the left kernel of the $m \times n$ Jacobian $df(x)$ is equal to the left kernel of the $m \times n$ Jacobian $df(y)$,
    i.e., $\ker df(x)^\top = \ker df(y)^\top$.
\end{restatable}

\subsection{Flexible-lattice case}

We begin with the following lemma regarding generic $\mathbb{Z}^d$-frameworks (see \Cref{subsec:coord}).

\begin{lemma}\label{l:flexible lattice sufficient}
    Let $(G,p,L)$ be a generic $\mathbb{Z}^d$-framework in $\mathbb{R}^d$.
    If $(G,p,L)$ has an equilibrium stress $\omega$ with $\dim\ker \mathcal{L}_{\mathbb{Z}^d}(G,\omega) = d+1$,
    then every equivalent $\mathbb{Z}^d$-framework is an affine transformation of $(G,p,L)$.
\end{lemma}

\begin{proof}	
    Choose an equivalent $\mathbb{Z}^d$-framework $(G,p',L')$.
    By applying \Cref{l:connelly lemma} to $f_G$, $(p,L)$ and $(p',L')$,
    we see that $(G,p,L)$ and $(G,p',L')$ have the same equilibrium stresses;
    in particular, $\omega$ is an equilibrium stress of both.
    Hence $(G,p',L')$ is an affine transformation of $(G,p,L)$ by \Cref{prop:afftrans}.
\end{proof}

Our next aim is to strengthen \Cref{l:flexible lattice sufficient} so that our equilibrium stress condition implies global rigidity.
The following result is similar to an argument given in \cite[Proposition 21.1]{Connelly2013}.

\begin{lemma}\label{l:conic inf rig}
    Let $(G,p,L)$ be a generic $\mathbb{Z}^d$-framework.
    If $(G,p,L)$ is infinitesimally rigid,
    then the edge directions of $(G,p,L)$ do not lie on a conic at infinity.
\end{lemma}

\begin{proof}
    We proceed by proving the contrapositive statement.
    Suppose that the edge directions of $(G,p,L)$ do lie on a conic at infinity with respect to the non-zero symmetric matrix $Q$.
    It follows from the spectral theorem for symmetric matrices that there exists an orthogonal $d \times d$ matrix $X$ such that $X^\top Q X = D$,
    where $D$ is the diagonal matrix with diagonal $(\lambda_1,\ldots, \lambda_d)$,
    where $\lambda_1 \leq \ldots \leq \lambda_d$ are the eigenvalues of $Q$.
    By scaling $Q$,
    we may suppose that $\lambda_d \leq 1$.
    For each $t \in [0,1]$,
    define $M_t$ to be the diagonal real-valued matrix with diagonal
    \begin{align*}
        \left(\sqrt{1- t\lambda_1}, ~\ldots ~, ~\sqrt{1 - t \lambda_d} \right),
    \end{align*}
    and fix $A_t = X^\top M_t X$.
    
    Define for each $t \in [0,1]$ the $\mathbb{Z}^d$-framework $(G,p_t,L_t)$ by setting $p_t(v) = A_tp(v)$ for each $v \in V$ and $L_t = A_t L$.
    We note that for each $t \in [0,1]$ we have
    \begin{align*}
        I_d - A^\top_t A_t = X^\top \left( I_d - M_t^\top M_t  \right) X = X^\top \left( I_d - M_t^2 \right) X = X^\top (tD) X = t Q.
    \end{align*}
    Hence each $(G,p_t,L_t)$ is equivalent to $(G,p,L)$ by \Cref{l:conic equiv}.
    Since each $A_t$ is not an orthogonal matrix and $(G,p,L)$ is affinely spanning (as it is generic),
    it follows that each $(G,p_t,L_t)$ is equivalent but not congruent to $(G,p,L)$ for all sufficiently small values of $t$.
    Hence $(G,p,L)$ is not locally rigid.
    By \Cref{t:gen rig equiv},
    it now follows that $(G,p,L)$ is not infinitesimally rigid.
\end{proof}

With this we can now prove the following.

\begin{theorem}\label{t:flexible lattice sufficient}
    Let $(G,p,L)$ be a generic infinitesimally rigid $\mathbb{Z}^d$-framework in $\mathbb{R}^d$.
    If $(G,p,L)$ has an equilibrium stress $\omega$ with $\dim\ker \mathcal{L}_{\mathbb{Z}^d}(G,\omega) = d+1$,
    then $(G,p,L)$ is globally rigid.
\end{theorem}

\begin{proof}	
    Choose any $\mathbb{Z}^d$-framework $(G,p',L')$ in $\mathbb{R}^d$ that is equivalent to $(G,p,L)$.
    By \Cref{l:flexible lattice sufficient},
    there exists a linear transform $A$ and a translational vector $x \in \mathbb{R}^d$ so that $L' = AL$ and $p'(v)= Ap(v) +x$ for every $v \in V$.
By \Cref{l:conic inf rig}, the edge directions of $(G, p, L)$ do not lie on a conic at infinity. Then \Cref{l:conic equiv} implies
\[
A^\top A - I_d = 0
\]
so $A$ is orthogonal.
Hence $(G, p', L')$ is congruent to $(G, p, L)$, proving global rigidity.
\end{proof}

Similar to the classical  framework situation,
the genericity condition is required for \Cref{t:flexible lattice sufficient}, even if the $\mathbb{Z}^d$-framework in question is infinitesimally rigid;
this can be shown by adapting known non-periodic frameworks with full rank equilibrium stresses (e.g., \cite[Figure 3.10(e)]{baker2025geometry}) to be periodic frameworks.
However,
the existence of an infinitesimally rigid $\mathbb{Z}^d$-framework with a full rank equilibrium stress implies that all sufficiently close generic $\mathbb{Z}^d$-frameworks will also have a full rank equilibrium stress.
Additionally, the existence of a single generic $\mathbb{Z}^d$-framework with a full rank equilibrium stress implies that \emph{all} generic $\mathbb{Z}^d$-frameworks have a full rank equilibrium stress.
This latter observation is an important step in proving that global rigidity is a generic property for periodic frameworks (\Cref{thm:ggr}).

\begin{proposition}\label{prop:generic stresses}
    Let $(G,p,L)$ be an infinitesimally rigid $\mathbb{Z}^d$-framework.
    If $(G,p,L)$ has an equilibrium stress $\omega$ with $\dim\ker \mathcal{L}_{\mathbb{Z}^d}(G,\omega) = d+1$,
	then every generic $\mathbb{Z}^d$-framework also has an equilibrium stress $\omega'$ with $\dim\ker \mathcal{L}_{\mathbb{Z}^d}(G,\omega') = d+1$.
\end{proposition}

Before we prove \Cref{prop:generic stresses}, we require the following concepts from real algebraic geometry, which will also be used in \Cref{sec:necessgen}:
\begin{itemize}
    \item We denote the real Zariski closure of a semi-algebraic set $M$ by $\overline{M}$.
    \item A point $x$ in an integral semi-algebraic set $M$ (see \Cref{app:int}) is a \emph{generic point of $M$} if its coordinates do not satisfy any algebraic equation with coefficients in $\mathbb{Q}$ besides those that are satisfied by every point in $M$.
    This definition ties into our previous definition of a generic point under the following observation:
    a point $x \in \mathbb{R}^n$ is a generic point if and only if it is a generic point of $\mathbb{R}^n$.
    \item For a homogeneous algebraic set $M \subset \mathbb{R}^n$, the \emph{dual variety} $M^*$ is the real Zariski closure of the set of all points $y \in \mathbb{R}^n$ which are a tangent vector to some smooth point $x \in M$.
    \item A vector $\phi \in \mathbb{R}^n$ is \emph{normal} to an algebraic set $M \subset \mathbb{R}^n$ at a smooth point $x$ if $T_x M \subset \ker \phi$.
    The \emph{conormal bundle} of an algebraic set $M \subset \mathbb{R}^n$ is the set
    \begin{equation*}
        C(M) := \overline{\left\{ (x,\phi) \in \mathbb{R}^n \times \mathbb{R}^n : x \in M \text{ smooth}, ~  \text{$\phi$ is normal to $M$ at $x$}  \right\}}.
    \end{equation*}
\end{itemize}

Using this language, it is easy to see that equilibrium stresses are (for the most part) normal vectors to the image of $f_G$.
The following lemma is an immediate adaption of \cite[Lemma 2.21]{gortler2010characterizing} to the periodic setting.

\begin{lemma}\label{lem:2.21}
    Let $(p,L) \in \mathcal{R}(G)$ be a realization where $f_G(p,L)$ is a smooth point of the semi-algebraic set $f_G(\mathcal{R}(G))$.
    If $\omega \in \mathbb{R}^E$ is normal to $f_G(\mathcal{R}(G))$ at $f_G(p,L)$,
    then $\omega$ is an equilibrium stress of $(G,p,L)$.
    If furthermore $(G,p,L)$ is infinitesimally rigid, then any equilibrium stress $\omega$ of $(G,p,L)$ is normal to $f_G(\mathcal{R}(G))$ at $f_G(p,L)$.
\end{lemma}

Each of the real (semi-)algebraic sets $f_G(\mathcal{R}(G))$, $\overline{f_G(\mathcal{R}(G))}$, $\overline{f_G(\mathcal{R}(G))}^*$ and $C(\overline{f_G(\mathcal{R}(G))})$ are homogenous and integral.
Additionally, since $\overline{f_G(\mathcal{R}(G))}$ is irreducible,
both $\overline{f_G(\mathcal{R}(G))}^*$ and $C(\overline{f_G(\mathcal{R}(G))})$ are also irreducible by \cite[Lemma 2.18]{gortler2010characterizing}.
We can bring equilibrium stresses into the picture using the following lemma that is an immediate adaption of \cite[Lemma 2.24]{gortler2010characterizing} to the periodic setting.

\begin{lemma}\label{lem:2.24}
    If $(G,p,L)$ is generic, then there exists an equilibrium stress $\omega$ of $(G,p,L)$ such that:
    \begin{enumerate}
        \item $(f_G(p,L),\omega)$ is a generic point of $C(\overline{f_G(\mathcal{R}(G))})$;
        \item $\omega$ is a generic point of $\overline{f_G(\mathcal{R}(G))}^*$.
    \end{enumerate}
\end{lemma}

With this, we are now ready to prove \Cref{prop:generic stresses}.

\begin{proof}[Proof of \Cref{prop:generic stresses}]
    We can use a similar technique to that described in the proof of \cite[Theorem 5]{ConnellyWhiteley2010} to find a generic 
    $\mathbb{Z}^d$-framework that has an equilibrium stress with nullity $d+1$.
    Hence, we may suppose $(G,p,L)$ is also generic.
    
    Fix the integral algebraic set
    \begin{equation*}
        Z := \left\{ \omega' \in \mathbb{R}^E : \dim \ker \mathcal{L}_{\mathbb{Z}^d}(G,\omega') > d+1 \right\}.
    \end{equation*}
    As $\omega \notin Z$ and $\overline{f_G(\mathcal{R}(G))}^*$ is irreducible,
    we have that $\dim \ker \mathcal{L}_{\mathbb{Z}^d}(G,\omega') = d+1$ for every generic point $\omega' \in \overline{f_G(\mathcal{R}(G))}^*$.
    The result now follows from \Cref{lem:2.24}.
\end{proof}

\subsection{Fixed-lattice case}

The case where the lattice is fixed is notably simpler and only requires the weaker property of $L$-genericity (see \Cref{subsec:coord}) since the coefficients of $f_{G,L}$ are contained in the smallest field generated by the rationals and the coefficients of $L$.

\begin{theorem}\label{t:fixed lattice sufficient}
    Let $(G,p,L)$ be an $L$-generic $\mathbb{Z}^d$-framework in $\mathbb{R}^d$.
    If $(G,p,L)$ has a fixed-lattice equilibrium stress $\omega$ with $\dim\ker \mathcal{L}(G,\omega) = 1$,
    then $(G,p,L)$ is fixed-lattice globally rigid.
\end{theorem}

\begin{proof}
    Choose an equivalent $\mathbb{Z}^d$-framework $(G,p',L)$.
    By applying \Cref{l:connelly lemma} to $f_{G,L}$, $p$ and $p'$,
    we see that $(G,p,L)$ and $(G,p',L)$ have the same fixed-lattice equilibrium stresses;
    in particular, $\omega$ is a fixed-lattice equilibrium stress of both.
    Hence $(G,p',L)$ is congruent to (and in fact a translated copy of) $(G,p,L)$ by \Cref{prop:fltransl}.
\end{proof}

The natural analog of \Cref{prop:generic stresses} is also true for the fixed-lattice case.

\begin{proposition}\label{prop:fixed lattice generic stresses}
    Let $(G,p,L)$ be a non-flat $\mathbb{Z}^d$-framework that is either fixed-lattice infinitesimally rigid or $L$-generic.
    If $(G,p,L)$ has a fixed-lattice equilibrium stress $\omega$ with $\dim\ker \mathcal{L}(G,\omega) = 1$,
	then for any non-singular lattice $L'$, every $L'$-generic $\mathbb{Z}^d$-framework has a fixed-lattice equilibrium stress $\omega'$ with $\dim\ker \mathcal{L}(G,\omega') = 1$.
\end{proposition}

\begin{proof}
    If $(G,p,L)$ is $L$-generic, then it is fixed-lattice infinitesimally rigid by \Cref{t:fixed lattice sufficient} and \Cref{t:gen rig equiv fl}.
    Similarly, if $(G,p,L)$ is fixed-lattice infinitesimally rigid,
    then we may replace it with an $L$-generic $\mathbb{Z}^d$-framework using a similar technique to that described in the proof of \cite[Theorem 5]{ConnellyWhiteley2010}.
    Affine transformations on $(p,L)$ do not change a $\mathbb{Z}^d$-framework's fixed-lattice equilibrium stresses (a consequence of \Cref{prop:suff1}\ref{prop:suff1item2}), nor do they alter the nullity of $R_L(G,p)$.
    Hence,
    it suffices to prove the result in the particular case where $L' = L$ and $(G,p,L)$ is both generic and fixed-lattice infinitesimally rigid.

    By \Cref{lem:extend}, there exist $\binom{d+1}{2}$ loops $F$ which we can add to $G$ and a vector $\mu \in \mathbb{R}^F$ so that $(G+F,p,L)$ is infinitesimally rigid and $(\omega,\mu)$ is an equilibrium stress of $(G+F,p,L)$.
    Fix the integral algebraic set
    \begin{equation*}
        Z := \left\{ (\omega', \mu') \in \mathbb{R}^{E \cup F} : \dim \ker \mathcal{L}(G,\omega') > 1 \right\}.
    \end{equation*}
    As $(\omega,\mu) \notin Z$ and $\overline{f_{G+F}(\mathcal{R}(G+F))}^*$ is irreducible,
    we have $\dim \ker \mathcal{L}(G,\omega') = 1$  for every generic point $(\omega',\mu') \in \overline{f_{G+F}(\mathcal{R}(G+F))}^*$.
    
    Now choose an $L$-generic $q \in (\mathbb{R}^d)^V$.
    Since $(G,q,L)$ is generic, \Cref{lem:2.24} guarantees that $(G,q,L)$ has an equilibrium stress $(\omega',\mu')$ which is a generic point of 
    $\overline{f_{G+F}(\mathcal{R}(G+F))}^*$.
    Hence $(G,q,L)$ has a fixed-lattice equilibrium stress $\omega'$ where $\dim \ker \mathcal{L}(G,\omega') = 1$, which concludes the proof.
\end{proof}

\section{Necessary global rigidity conditions for generic periodic frameworks}\label{sec:necessgen}

In this section we prove the necessary condition for \Cref{mainthm:flexlattice}.
This proof follows essentially the same method to that given by Gortler, Healy and Thurston's landmark result for classical non-periodic frameworks \cite{gortler2010characterizing}.

\subsection{Flexible-lattice case}
 
To prove the necessary condition for \Cref{mainthm:flexlattice}, we will prove the following analog of \cite[Theorem 1.14]{gortler2010characterizing}:

\begin{theorem}\label{thm:ght}
    Let $(G,p,L)$ be a generic $\mathbb{Z}^d$-framework.
    If $(G,p,L)$ has no equilibrium stress $\omega$ with $\dim\ker \mathcal{L}_{\mathbb{Z}^d}(G,\omega) = d+1$,
    then $(G,p,L)$ is not globally rigid.
\end{theorem}

To simplify the following lemmas,
we will fix the following assumptions throughout the subsection.
We pick $G$ to be a $\mathbb{Z}^d$-gain graph with a generic realization $(p,L)$ where every equilibrium stress $\omega$ satisfies $\dim\ker \mathcal{L}_{\mathbb{Z}^d}(G,\omega) > d+1$ (this implies that $|V| \ge 2$).
We may assume that $(G,p,L)$ is locally rigid,
as otherwise the $\mathbb{Z}^d$-framework is trivially not globally rigid.
Hence, by \Cref{t:gen rig equiv}, $(G,p,L)$ will be infinitesimally rigid throughout.

We additionally introduce the following terminology we use throughout:
\begin{enumerate}
    \item We denote by $\Euc(d)$ the isometries of $\mathbb{R}^d$, and we denote by $\Trans(d)$ the group of translations of $\mathbb{R}^d$.
    Both groups have a natural action on $\mathcal{R}(G)$.
    \item For any equilibrium stress $\omega$ of $(G,p,L)$, we fix $A(\omega) \subset \mathcal{R}(G)$ to be the linear subspace of $\mathbb{Z}^d$-frameworks $(q,L')$ where $\omega$ is also an equilibrium stress of $(G,q,L')$.
    It follows from \Cref{prop:suff1}\ref{prop:suff1item1} that $A(\omega) = \prod_{i=1}^d \ker \mathcal{L}_{\mathbb{Z}^d}(G,\omega)$.
    \item We fix $\mathcal{L}(\omega)$ to be the smallest linear space containing $f_G(A(\omega))$.
    Since $f_G$ is positive homogenous (i.e., $f_G(\lambda p',\lambda L') = |\lambda| f_G(p',L')$) and $A(\omega)$ is a linear space, the set $f_G(A(\omega))$ is a closed semi-algebraic subset of $\mathcal{L}(\omega)$.
    \item We fix the map
    \begin{align*}
        f_\omega : A(\omega)/\Euc (d) \rightarrow \mathcal{L}(\omega), ~ [(p',L')] \rightarrow f_G(p',L').
    \end{align*}
    Since $G$ is connected (a consequence of $(G,p,L)$ being locally rigid),
    a similar method to \cite[Lemma 2.34]{gortler2010characterizing} gives that $f$ is a proper map (i.e., the preimage of any compact set is itself compact).
\end{enumerate}

Our main aim is to be able to show that, for a ``generic'' equilibrium stress, the map $f_{\omega}$ has a well-defined concept of degree modulo 2 that corresponds to the number of equivalent but non-congruent realizations modulo isometries.
From this, the basic idea is to show that this degree must be even, since this will then imply $(G,p,L)$ is not globally rigid.

\subsubsection{Domain of \texorpdfstring{$f_\omega$}{f-omega}}

We first deal with the domain of $f_\omega$.

To quote Gortler, Healy and Thurston,
a \emph{smooth stratified space} is \emph{``loosely speaking, a space which is decomposed into smooth manifolds of differing dimensions, limiting onto each other in a nice way''}. For a more rigorous definition, see \cite{Pflaum}.

With this geometric concept in mind, we now require the following two results regarding stratified spaces, which can be seen to be periodic analogs to  \cite[Proposition 2.12]{gortler2010characterizing} and  \cite[Proposition 2.13]{gortler2010characterizing} respectively.

\begin{proposition}\label{prop:2.12}
The quotient space $\mathcal{R}(G)/Euc(d)$ is a smooth stratified space
with singularities of codimension $|V|$ or higher. Furthermore, the singularities occur at classes of frameworks that are not affinely spanning.
\end{proposition}

\begin{proof}
    The proof follows from \cite[Proposition 2.12]{gortler2010characterizing} with minor modifications: now the dimension of $\mathcal{R}(G)/\Euc(d)$ is $d|V| - \binom{d+1}{2} + d^2$ and the dimension of the subspace of all equivalence classes of not affinely spanning realizations in $\mathcal{R}(G)/\Euc(d)$ (which are exactly the singularities of the space) is $(d-1)|V| - \binom{d}{2} + d(d-1)$.
    Hence the singularities must have codimension 
    \begin{equation*}
        \left(d|V| - \binom{d+1}{2} + d^2 \right)- \left((d-1)|V| - \binom{d}{2} + d(d-1) \right) = |V| -d +d = |V|
    \end{equation*}
    or higher.
\end{proof}

\begin{proposition}\label{prop:2.13}
    Let $\omega$ be an equilibrium stress of $(G,p,L)$ and fix $k= \dim \ker \mathcal{L}_{\mathbb{Z}^d}(G,\omega)$.
    Then $A(\omega)/\Euc(d)$  is a smooth stratified space with singularities of codimension at least 2. Furthermore, the singularities occur at classes of frameworks that are not affinely spanning.
\end{proposition}

\begin{proof}
    The proof follows from \cite[Proposition 2.13]{gortler2010characterizing} with minor modifications: now the dimension of $A(\omega)/\Euc(d)$ is $dk - \binom{d+1}{2}$ and the dimension of the subspace of all equivalence classes of not affinely spanning realizations in $A(\omega)/\Euc(d)$ (which are exactly the singularities) is $(d-1)k - \binom{d}{2}$.
    Hence the singularities must have codimension 
    \begin{equation*}
        \left(dk - \binom{d+1}{2}\right)- \left((d-1)k - \binom{d}{2}\right) = k-d \geq 2
    \end{equation*}
    or higher, with the final inequality stemming from our assumption that $k \geq d+2$.
\end{proof}

\subsubsection{Codomain of \texorpdfstring{$f_\omega$}{f-omega} and generic equilibrium stresses}

We now show that when $\omega$ is generic, the image of $f_\omega$ has the same dimension as its codomain.

The following lemma is an immediate adaption of 
\cite[Lemma 2.22]{gortler2010characterizing} 
to the periodic setting.

\begin{lemma}\label{lem:2.22}
    Let $\omega$ be an equilibrium stress of a generic framework $(G,p',L')$,
    and fix the set
    \begin{align*}
        B^o(\omega) := \left\{ \ell \in f_G(\mathcal{R}(G)) \text{ smooth} : \omega \text{ is normal to $f_G(\mathcal{R}(G))$ at $\ell$}   \right\}.
    \end{align*}
    Then the Euclidean closure of $B^o(\omega)$ is $f_G(A(\omega))$.
\end{lemma}

The following result is an immediate adaption of \cite[Proposition 2.23]{gortler2010characterizing} to the periodic setting that uses the previous result and \Cref{lem:2.21}.

\begin{proposition}\label{proposition:2.23}
    Let $\omega$ be an equilibrium stress of a generic framework $(G,p',L')$.
    If $\omega$ is generic in $\overline{f_G(\mathcal{R}(G))}^*$,
    then $\dim f_G(A(\omega)) = \dim \mathcal{L}(\omega)$.
\end{proposition}

\subsubsection{Proof of \texorpdfstring{\Cref{thm:ght}}{Gortler-Healy-Thurston analogoue}}

We are now almost ready to prove \Cref{thm:ght}.
We first need the following key result that describes the concept of ``degree'' that we need.
Given a map $f : X \rightarrow Y$ between smooth stratified spaces, we recall that a \emph{regular value} is any point $y \in Y$ where for each $x \in f^{-1}(y)$, the point $x$ is a differentiable point of $f$ and $\rank df(x) \geq \rank df(x')$ for all other differentiable points $x' \in X$.

\begin{corollary}[{\cite[Corollary 2.36]{gortler2010characterizing}}]\label{cor:2.36}
    If $X$ is a smooth stratified space with singularities of codimension
    at least 2, $Y$ is a smooth, connected manifold of the same dimension as $X$,
    and $f: X \rightarrow Y$ is a proper, smooth map, then there is an element $\deg f \in \mathbb{Z}/2\mathbb{Z}$ which is invariant
    under proper isotopies of $f$. 
    Moreover, $\deg f$ is equal to the number of preimages of any regular value in $Y$, taken modulo 2.
\end{corollary}

If a map $f$ satisfies the criteria given in \Cref{cor:2.36}, we say that $\deg f$ is the \emph{mod-two degree} of $f$.

The following two lemmas are analogs of \cite[Lemma 2.38]{gortler2010characterizing} and \cite[Lemma 2.39]{gortler2010characterizing} respectively.

\begin{lemma}\label{lem:2.38}
    Let $\omega$ be an equilibrium stress of $(G,p,L)$ that is generic in $\overline{f_G(\mathcal{R}(G))}^*$.
    Then $f_\omega$ has mod-two degree 0.
\end{lemma}

\begin{proof}
    By \Cref{prop:2.13}, the domain of $f_\omega$ is a smooth stratified space with singularities of codimension 2.
    Since the image of $f_\omega$ is $f_G(A(\omega))$,
    we have by \Cref{proposition:2.23} that the image of $f_\omega$, here denoted $\im f_\omega$, has the same dimension as $\mathcal{L}(\omega)$.
    As $(G,p,L)$ is infinitesimally rigid, the map $f_\omega$ is an immersion on some open dense subset of its domain.
    Hence, we have that
    \begin{equation*}
        \dim \mathcal{L}(\omega) = \dim \im f_\omega = \dim A(\omega)/\Euc (d).
    \end{equation*}
    By \Cref{cor:2.36},
    the map $f_\omega$ has a well-defined mod-two degree.
    Finally, since the image of $f_{\omega}$ is contained in $\mathbb{R}^E_{\geq 0}$ but $\mathcal{L}(\omega)$ contains points outside of $\mathbb{R}^E_{\geq 0}$,
    there exists a regular value $\lambda \in \mathcal{L}(\omega)$ such that $f_G^{-1}(\lambda) = \emptyset$.
    Hence, $f_\omega$ has mod-two degree 0.
\end{proof}

\begin{lemma}\label{lem:2.39}
    $f_G(p,L)$ is a regular value of both $f_G$ and $f_\omega$.
\end{lemma}

\begin{proof}
    As $(G,p,L)$ is infinitesimally rigid, the rank of $df_G(p,L)$ is maximal over all points in $\mathcal{R}(G)$ (and hence also for $A(\omega)$).
    By \Cref{l:connelly lemma},
    every point $(p',L') \in \mathcal{R}(G)$ where $f_G(p',L') = f_G(p,L)$ is contained in $A(\omega)$ and also has maximal rank;
    moreover, since $(G,p',L')$ is also infinitesimally rigid (and hence affinely spanning), the equivalence class $[(p',L')]$ is not contained in a singularity of $\mathcal{R}(G)/\Euc(d)$ or $A(\omega)/\Euc(d)$, and the derivative $df_{\omega}(p',L)$ is injective (and hence has maximal rank).
    Hence,
    $f_G(p,L)$ is a regular value of both $f_G$ and $f_\omega$.
\end{proof}

With this final lemma, we are now ready to prove \Cref{thm:ght}

\begin{proof}[Proof of \Cref{thm:ght}]
    By \Cref{lem:2.24},
    there exists an equilibrium stress $\omega$ of $(G,p,L)$ that is a generic point of $\overline{f_G(\mathcal{R}(G))}^*$.
    By \Cref{lem:2.38},
    $f_\omega$ has mod-two degree 0.
    As $f_G(p,L)$ is a regular value of $f_\omega$ (\Cref{lem:2.39}),
    the number of points in $f^{-1}_\omega(f_G(p,L))$ is positive (as it contains the equivalence class $[(p,L)]$) and even.
    Hence, there exists a $\mathbb{Z}^d$-framework $(G,p',L')$ that is equivalent but not congruent to $(G,p,L)$ such that $(p',L') \in A(\omega)$.
    This now implies that $(G,p,L)$ is not globally rigid as required.
\end{proof}

\subsection{Fixed-lattice case}\label{subsec:fixednec}

Analogous to the previous section, to prove the necessary condition for \Cref{mainthm:fixedlattice} we only need to prove the following:

\begin{theorem}\label{thm:ghtfixed}
    Let $(G,p,L)$ be a generic $\mathbb{Z}^d$-framework.
    If $(G,p,L)$ has no fixed-lattice equilibrium stress $\omega$ with $\dim\ker \mathcal{L}(G,\omega) = 1$,
    then $(G,p,L)$ is not fixed-lattice globally rigid.
\end{theorem}

Our method for proving \Cref{thm:ghtfixed} is to shift our fixed-lattice framework problem into the more general flexible-lattice framework setting so as to utilize \Cref{thm:ght}.
For this, we first require the following result:

\begin{lemma}\label{lem:fixedhassamerank}
    Let $\omega$ be an equilibrium stress of a non-flat $\mathbb{Z}^d$-framework $(G,p,L)$,
    let $G'$ be the graph formed from $G$ by removing all loops and let $\omega'$ be the restriction of $\omega$ to $G'$.
    Then $\omega'$ is a fixed-lattice equilibrium stress of $(G',p,L)$ and 
    \begin{equation*}
        \rank \mathcal{L}_{\mathbb{Z}^d}(G,\omega) = \rank \mathcal{L}(G,\omega) = \rank \mathcal{L}(G',\omega').
    \end{equation*}
\end{lemma}

To prove \Cref{lem:fixedhassamerank}, we require the following simple lemma:

\begin{lemma}\label{lem:galeduality}
    Let $A \in \mathbb{R}^{d \times m}$, $B \in \mathbb{R}^{d \times d}$, $C \in \mathbb{R}^{m \times n}$, $D \in \mathbb{R}^{d \times n}$ be matrices satisfying the following equation:
    \begin{equation*}
        \begin{bmatrix}
            A & B
        \end{bmatrix}
        \begin{bmatrix}
            C \\
            D
        \end{bmatrix}
        = A C + BD = \mathbf{0}_{d \times n}.
    \end{equation*}
    If the matrix $B$ is non-singular, then the rows of $D$ are contained in the row span of $C$.
\end{lemma}

\begin{proof}
    The result follows from the observation that $D= (-B^{-1}A) C$.
\end{proof}

\begin{proof}[Proof of \Cref{lem:fixedhassamerank}]
    We first observe that
    \begin{equation}\label{eq:gprime1}
        \mathcal{L}(G,\omega) = \mathcal{L}(G',\omega'), \qquad M(G) {\rm diag}(\omega) I(G) = M(G') {\rm diag}(\omega') I(G').
    \end{equation}
    Since $\omega$ satisfies the equation given in \Cref{prop:suff1}\ref{prop:suff1item2},
    it follows from \cref{eq:gprime1} that 
    \begin{equation*}
        P\mathcal{L}(G',\omega') + L M(G') {\rm diag}(\omega') I(G') = P\mathcal{L}(G,\omega) + L M(G) {\rm diag}(\omega) I(G) = \mathbf{0}_{d \times (|V|+d)}.
    \end{equation*}
    Hence $\omega'$ is a fixed-lattice equilibrium stress of $(G',p,L)$.
    
    By \Cref{prop:suff1}\ref{prop:suff1item1},
    we have $[P ~ L] \mathcal{L}_{\mathbb{Z}^d}(G,\omega) = \mathbf{0}_{d \times (|V|+d)}$.
    As $L$ is non-singular, it follows from \Cref{lem:galeduality} that the first $|V|-d$ rows of $\mathcal{L}_{\mathbb{Z}^d}(G,\omega)$ span the entire row space of the matrix.
    Since $\mathcal{L}_{\mathbb{Z}^d}(G,\omega)$ is symmetric, it follows that the rank of $\mathcal{L}_{\mathbb{Z}^d}(G,\omega)$ is equal to the rank of the matrix
    \begin{equation*}
        X =
        \begin{bmatrix}
            \mathcal{L}(G,\omega) \\
            M(G') {\rm diag}(\omega') I(G')
        \end{bmatrix}.
    \end{equation*}
    Since $[P ~ L] X = \mathbf{0}_{d \times |V|}$, it again follows from \Cref{lem:galeduality} that the first $|V|-d$ rows of $X$ span the entire row space of the matrix.
    Hence $\rank \mathcal{L}_{\mathbb{Z}^d}(G,\omega) = \rank X = \rank \mathcal{L}(G,\omega)$.   
\end{proof}

\begin{proof}[Proof of \Cref{thm:ghtfixed}]
    Pick an arbitrary vertex $u \in V$.
    From this, we construct a set of $\binom{d+1}{2}$ loops $F$ at $u$ such that the set
    \begin{equation*}
        \Gamma := \left\{ \gamma \gamma^\top : (u,u,\gamma) \in F\right\}
    \end{equation*}
    is a basis of the space of $d \times d$ symmetric matrices.
    By \Cref{lem:extend},
    there exists a bijection from the left kernel of $R(G+F,p,L)$ to $R_L(G,p)$ given by $(\omega_e)_{e \in E \cup F} \mapsto (\omega_e)_{e \in E}$.
    Hence by \Cref{lem:fixedhassamerank},
    we have $\dim \ker \mathcal{L}_{\mathbb{Z}^d}(G+F,\omega) > d+1$ for all equilibrium stresses $\omega$ of $(G+F,p,L)$.
    By \Cref{thm:ght},
    there exists a $\mathbb{Z}^d$-framework $(G+F,q,M)$ that is equivalent but non-congruent to $(G+F,p,L)$.
    Using that $\langle A, B \rangle := {\rm Tr}(A B)$ is an inner product on the space of real symmetric matrices, each loop $(u,u,\gamma) \in F$ describes the following equality:
    \begin{gather*}
        \gamma^\top L^\top L \gamma = \|p(u) + L \gamma - p(u) \|^2 = \|q(u) + M \gamma - q(u) \|^2 = \gamma^\top M^\top M \gamma \\
        \Updownarrow \\
        \gamma^\top (L^\top L - M^\top M) \gamma = 0 \\
        \Updownarrow \\
        \left\langle (L^\top L - M^\top M) ~ , ~ \gamma \gamma^\top \right\rangle = {\rm Tr}\left((L^\top L - M^\top M) \gamma \gamma^\top \right) = 0.
    \end{gather*}
    Since the set $\Gamma$ is a basis of the space of $d\times d$ symmetric matrices,
    it follows that $L^\top L = M^\top M$.
    This in turn implies that $L = A M$ for some orthogonal matrix $A$.
    Thus, if we set $p' = (A q(v))_{v \in V}$, then $(G,p',L)$ is equivalent but not congruent to $(G,p,L)$.
    Hence $(G,p,L)$ is not fixed-lattice globally rigid.    
\end{proof}

\section{Sufficient conditions for the global rigidity of tensegrities under volume constraints}\label{sec:boundedlat}

Global rigidity under lattice representations with  a volume constraint is, in general, not a generic property. For example, if there exists a generic (fully flexible lattice) locally rigid, but not globally rigid $\mathbb{Z}^d$-framework $(G,p,L)$ that becomes globally rigid once a volume lower bound is imposed on the lattice, then this framework is equivalent to a generic $\mathbb{Z}^d$-framework $(G,p',L')$ with a lattice of smaller volume, and so $(G,p',L')$ will not be globally rigid with the addition of a volume lower bound on the lattice. 

Nevertheless, in this section we establish a sufficient stress matrix condition for the global rigidity of periodic tensegrities where the volume of the lattice is given a lower bound.

\subsection{Rigidity of frameworks under volume constraints}

Recall that for a non-flat framework $(G,p,L)$, the lattice $L$ is nonsingular and that ${\cal R}^*(G)$ denotes the set of non-flat realizations of $G$. 
In order to study the (global) rigidity of frameworks subject to a volume constraint for the fundamental domain of the lattice, we consider the  \emph{extended measurement map}  
$f_G^{\rm vol}: {\cal R}^*(G)\rightarrow \mathbb{R}^{E} \times \mathbb{R}$ given by
\begin{equation*}
f_G^{\rm vol}(p,L)=
\begin{bmatrix}
    f_G(p,L) \\
    -\log |\det(L)|
\end{bmatrix}.
\end{equation*}
For a technical reason which we will see later, there is a minus sign in the last coordinate.
The extended rigidity matrix $R^{\rm vol}(G,p,L)$ is now defined as 
$\frac{1}{2}df_G^{\rm vol}(p,L)$, where $df_G^{\rm vol}(p,L)$ is the Jacobian of $f_G^{\rm vol}$ evaluated at $(p,L)$.

\begin{remark}
    For dimension $d=2$, the rigidity matrix $R^{\rm vol}(G,p,L)$ is the same as the one introduced by Malestein and Theran for  periodic frameworks in the plane under a unit-area lattice representation \cite{mtfixedarea}.
\end{remark}

To the best of our knowledge, the following concept of a $\lambda$-equilibrium stress is new.
Let $\lambda\in \mathbb{R}$, and assume that $L$ is nonsingular (i.e., $(G,p,L)$ is non-flat).
An edge weight $\omega:E\rightarrow \mathbb{R}$ is said to be a \emph{$\lambda$-equilibrium stress} of a $\mathbb{Z}^d$-framework $(G,p,L)$ if 
it satisfies \cref{eq:1} (i.e., it is a fixed-lattice equilibrium stress) and 
\begin{equation}\label{eq:3}
    \sum_{e = (u,v,\gamma)\in E} \omega(e) (p(v) + L\gamma -p(u)) \gamma^{\top}=\lambda L^{-\top},
\end{equation} 
where $L^{-\top}$ denotes the transpose of the inverse of $L$.
If $(G,p,L)$ is a tensegrity, we also require that $\omega$ must be a proper edge weighting.

\begin{proposition}\label{prop:lambda_stress}
An edge weight $\omega:E(G)\rightarrow \mathbb{R}$ is a $\lambda$-equilibrium stress of a non-flat $\mathbb{Z}^d$-framework $(G,p,L)$ if and only if $(\omega^{\top}, \lambda)$ is in the left kernel of  $R^{\rm vol}(G,p,L)=0$.
\end{proposition}

\begin{proof}
    This follows from the well-known fact that if $g(L):=\log |\det(L)|$ is considered  a function over ${\rm GL}(d)$, then its gradient is $L^{-\top}$.
\end{proof}

We now consider the global rigidity of $\mathbb{Z}^d$-frameworks under an additional volume constraint for the lattice, 
where the \emph{volume} of $L$ (or $(G,p,L)$) is defined to be $|\det(L)|$.
Our main theorem is the following. 

\begin{theorem}\label{thm:2}
    Let $(G,p,L)$ be a non-flat $\mathbb{Z}^d$-tensegrity with volume equal to one.
    Then $(G,p,L)$ is globally rigid over all tensegrities with volume at least one if there is a positive number $\lambda$ and a $\lambda$-equilibrium stress $\omega$ of $(G,p,L)$ such that 
    \begin{itemize}
    \item $\dim \ker \mathcal{L}_{\mathbb{Z}^d} (G,\omega) =1$, and 
    \item $\mathcal{L}_{\mathbb{Z}^d} (G,\omega)$ is positive semidefinite.
    \end{itemize}
\end{theorem}

The proof is done by analyzing the following optimization problem over the realization space ${\cal R}(G)$
for a $\mathbb{Z}^d$-gain graph $G$ with edge weight $\omega:E\rightarrow \mathbb{R}$:
\[
\begin{array}{lll}
\prob & {\rm minimize} & {\cal E}_{G,\omega}(p,L) \\
& {\rm subject\ to} & \log |\det  L|\geq 0
\end{array}
\]
If $\mathcal{L}_{\mathbb{Z}^d} (G,\omega)$ is positive semidefinite then ${\cal E}_{G,\omega}$ is a convex function. 
However, the condition $ \log |\det  L|$ is not concave over ${\cal R}^*(G)$ or over the set of matrices with positive determinant. (It is concave over the set of positive semidefinite matrices but not over the set of matrices with positive determinant). 
So the optimization problem $\prob$ is not a convex programming problem, and there is no general theory to characterize its global minimizer. Nevertheless, we shall prove that any local minimizer of $\prob$ is a global minimum 
and that there is a unique optimal solution (up to isometries).

To give a precise statement, we shall introduce the notion of KKT points  
based on the  Karush-Kuhn-Tucker theorem, or KKT theorem in short.
\begin{theorem}[Karush-Kuhn-Tucker theorem]\label{thm:KKT}
Let $f:\mathbb{R}^k\rightarrow \mathbb{R}$ and $g:\mathbb{R}^k\rightarrow \mathbb{R}^m$ be continuously differentiable
at a point $x^*\in \mathbb{R}^k$, and suppose that the Jacobian $Jg(x^*)$ of $g$ at $x^*$ is full row rank.
If $x^*$ is a local minimizer of the optimization problem: 
\[
\begin{array}{lll}
{\rm (P)} & {\rm minimize} & f(x) \\
& {\rm subject\ to} & g(x)\leq 0,
\end{array}
\]
then there is a multiplier $\lambda\in \mathbb{R}^m$ satisfying the following conditions:
\begin{description}
\item[stationarity condition:] $\nabla f(x^*)+Jg(x^*)^{\top} \lambda=0$; 
\item[complementary slackness:] $\lambda_i g_i(x^*)=0$ for each $1\leq i\leq m$;
\item[sign condition:] $\lambda_i\geq 0$ for each $1\leq i\leq m$. 
\end{description} 
\end{theorem}

We will apply this theorem to $\prob$. Note that this requires a sign change in the condition for $\log |\det L|$ to put $\prob$ in standard form.

A point $x\in \mathbb{R}^k$ for which there exists a complementary $\lambda$ satisfying the three conditions of \Cref{thm:KKT} is called a \emph{KKT point}  in the problem (P). The KKT theorem says that any local minimizer is a KKT point provided that $Jg(x^*)$ is of  full row rank. 
We comment here that the converse is not true in general.

We now prove the following remarkable property of  $\prob$.

\begin{theorem}\label{thm:3}
Let $G$ be a $\mathbb{Z}^d$-gain graph with edge weight $\omega$. 
If $\dim \ker \mathcal{L}_{\mathbb{Z}^d}(G,\omega)=1$ and $\mathcal{L}_{\mathbb{Z}^d}(G,\omega)$ is positive semidefinite, then the optimization problem $\prob$ has a unique KKT point $(p^*,L^*)$ 
up to isometries, and it is a  global minimizer of $\prob$. 
\end{theorem}

The proof of \Cref{thm:3} will be given in the next subsection.
The special case when the edge weight is positive provides an alternative proof to (the Euclidean case of) Theorem 2 in \cite{KotaniSunada2001}.

\begin{corollary}[Euclidean case of Kotani and Sunada]
Let $G$ be a connected $\mathbb{Z}^d$-gain graph with positive edge weight $\omega$.
Suppose that the rank of $G$ is $d$. 
Then the optimization problem $\prob$ has a global minimizer  $(p^*,L^*)$ which is unique 
up to isometries.
\end{corollary}
\begin{proof}
If the rank of $G$ is $d$, then its  incidence matrix $I(G)$ has rank $|V|-1+d$ by \Cref{prop:1}.
Hence, as the edge weight is positive, $\mathcal{L}_{\mathbb{Z}^d}(G,\omega)=I(G)^{\top}{\rm diag}(\omega)I(G)$ is positive semidefinite with rank $|V|-1+d$. Thus \Cref{thm:3} applies.
\end{proof}

Our sufficient condition for global rigidity (\Cref{thm:2}) easily follows from \Cref{thm:3} and the following lemma.
Note that $\prob$ has only one constraint, and hence a multiplier $\lambda$ will be a scalar.

\begin{lemma}\label{lem:vol1}
Let $(G,p,L)$ be a non-flat $\mathbb{Z}^d$-framework and let $\lambda\in \mathbb{R}$. The following are equivalent:
\begin{enumerate}
\item \label{lem:vol1item1} $(p,L)$ and $\lambda$ satisfy the stationarity condition of $\prob$.
\item \label{lem:vol1item2} $[P ~ L] \mathcal{L}_{\mathbb{Z}^d}(G,\omega) = \lambda [ \mathbf{0}_{d \times |V|} ~  L^{-\top}].$
\item \label{lem:vol1item3} $\omega$ is a $\lambda$-equilibrium stress of $(G,p,L)$.
\end{enumerate}
\end{lemma}

\begin{proof}
The equivalence of \ref{lem:vol1item1} and \ref{lem:vol1item2} can be checked by a direct calculation, using the fact that 
the gradient of $\log \det(\cdot)$ at $L$ is $L^{-\top}$.

The equivalence of \ref{lem:vol1item2} and \ref{lem:vol1item3} can be checked by using exactly the same argument as in the proof of \Cref{prop:suff1}.
\end{proof}

\begin{proof}[Proof of \Cref{thm:2}]
Suppose that $(G,p,L)$ admits a $\lambda$-equilibrium stress $\omega$ for some $\lambda>0$ such that $\dim \ker \mathcal{L}_{\mathbb{Z}^d}(G,\omega)=1$ and $\mathcal{L}_{\mathbb{Z}^d}(G,\omega)$ is positive semidefinite.
Consider $\prob$ for $G$ and $\omega$.
By \Cref{lem:vol1}, $(p,L)$ and $\lambda$ satisfy the stationarity condition of $\prob$. Moreover, the complementary slackness condition holds for $\prob$ since $|\det(L)|=1$. Hence $(p,L)$ is a KKT point, and by \Cref{thm:3} $(p,L)$ is a global minimizer.

Consider any $(G,p',L')$ dominated by $(G,p,L)$.
Then (by the definition of equilibrium stresses of tensegrities) we have 
\begin{equation*}
    \omega(e)\|p'(v) + L' \gamma -p'(u)\|^2 \leq \omega(e)\|p(v) + L \gamma -p(u)\|^2
\end{equation*}
for each edge $e=(u,v,\gamma)\in E$.
Hence from \cref{eq:4.2} we have ${\cal E}_{G,\omega} (p',L')\leq {\cal E}_{G,\omega} (p,L)$.
As $(p,L)$ is a global minimizer, $(p', L')$ is also a global minimizer.
Therefore, \Cref{thm:3} further implies that $(G,p',L')$ is congruent to $(G,p,L)$.
\end{proof}
  

\subsection{Proof of \texorpdfstring{\Cref{thm:3}}{Theorem}}
The proof is done by a sequence of lemmas.
We focus on $\prob$ for a given $G$ and $\omega$.

\begin{lemma} \label{lem:vol2}
If $\dim \ker \mathcal{L}_{\mathbb{Z}^d}(G,\omega)=1$ and $\mathcal{L}_{\mathbb{Z}^d}(G,\omega)$ is positive semidefinite, then $\prob$ has a global minimizer.
\end{lemma}

\begin{proof}
Since the energy function ${\cal E}_{G,\omega}$ as well as the volume is invariant by translations,  
we may restrict our attention to realizations $(p,L)$ with $p(v_1)= \mathbf{0}$ for some vertex $v_1 \in V$.
With this restriction, ${\cal E}_{G,\omega}$ becomes a positive definite quadratic form.
Indeed, let $\Omega$ be the matrix obtained from $\mathcal{L}_{\mathbb{Z}^d}(G,\omega)$ by removing the row and column associated to $v_1$.
Then the energy function ${\cal E}_{G,\omega}$ restricted to the realizations $(p,L)$ with $p(v_1)=\mathbf{0}$ is the quadratic form of $\Omega\otimes I_d$.
Since $\dim \ker \mathcal{L}_{\mathbb{Z}^d}(G,\omega)=1$ and $\mathcal{L}_{\mathbb{Z}^d}(G,\omega)$ is positive semidefinite, $\Omega$ is positive definite, implying that 
$\Omega\otimes I_d$ is positive definite.

Since the restriction of the energy function ${\cal E}_{G,\omega}$ to the realizations with $v_1$ at $\mathbf{0}$ is a quadratic form described by $\Omega\otimes I_d$, the lower level set $\{(p,L): {\cal E}_{G,\omega}(p,L)\leq \delta, ~ p(v_1) = \mathbf{0}\}$ is bounded for each $\delta \geq 0$. Note that the optimal value for $\prob$ does not change even if the feasible region of $\prob$ is restricted to a sufficiently large compact set.
So since the feasible region of the  restricted problem is compact, the optimal value is attained. 
\end{proof}

\begin{lemma}\label{lem:vol4}
Suppose that $\dim \ker \mathcal{L}_{\mathbb{Z}^d}(G,\omega)=1$ and $\mathcal{L}_{\mathbb{Z}^d}(G,\omega)$ is positive semidefinite.
Then for any KKT point $(p,L)$ with multiplier $\lambda\geq 0$, the following hold.
\begin{enumerate}
\item $|\det(L)|=1$;
\item for the coordinate matrix $P$ of $p$, 
$[P ~ L]  \mathcal{L}_{\mathbb{Z}^d}(G,\omega) [ P ~ L]^{\top}=\lambda I_d$, where $I_d$ is the $d\times d$ identity matrix.
\end{enumerate}
\end{lemma}
\begin{proof}
Suppose that $|\det(L)|>1$. Then the complementary slackness condition implies that $\lambda=0$.
As $(p,L)$ satisfies the stationarity condition, \Cref{lem:vol1} further implies 
$[P\ L] \mathcal{L}_{\mathbb{Z}^d}(G,\omega)=0$. Since $\dim \ker   \mathcal{L}_{\mathbb{Z}^d}(G,\omega)=1$ with $\hat{\bm 1}\in \ker  \mathcal{L}_{\mathbb{Z}^d}(G,\omega)$, 
this in turn implies that every row vector of $[P\ L]$ is a scalar multiple of $\hat{\bm 1}^{\top}$.
This in particular implies that $L=\mathbf{0}_{d \times d}$, contradicting $|\det(L)|\neq 0$.
Hence (i) holds.

(ii) directly follows from \Cref{lem:vol1} since 
\[
[P ~ L] \mathcal{L}_{\mathbb{Z}^d}(G,\omega) [P ~ L]^\top = \lambda \left[ \mathbf{0}_{d \times |V|} ~  L^{-\top} \right] 
\begin{bmatrix}P^{\top} \\ L^{\top}\end{bmatrix}=
\lambda I_d. \qedhere
\]
\end{proof}

\begin{lemma}\label{lem:vol5}
    Suppose that $\dim \ker \mathcal{L}_{\mathbb{Z}^d}(G,\omega)=1$.
    Any KKT points can be transformed to each other by an affine transformation.
\end{lemma}

\begin{proof}
    We split $\mathcal{L}_{\mathbb{Z}^d}(G,\omega)$ into two blocks as follows. Let $\mathcal{L}_{\mathbb{Z}^d}(G,\omega)=[\Omega_L\ \Omega_R]$, where 
    $\Omega_L$ has size $(|V|+d)\times n$ and $\Omega_R$ has size $(|V|+d)\times d$.
    
    Consider any KKT point $(p,L)$ with multiplier $\lambda$. 
    Then 
    $[P ~ L] \mathcal{L}_{\mathbb{Z}^d}(G,\omega) =\lambda [ \mathbf{0}_{d \times |V|} ~ L^{-\top}]$
    holds by \Cref{lem:vol1}.
    This in turn implies that each row vector of $[P ~ L]$ is in the left kernel of $\Omega_L$.
    Since $\rank \mathcal{L}_{\mathbb{Z}^d}(G,\omega) =|V|-1+d$, the rank of $\Omega_L$ is $|V|-1$. 
    (Note that $\hat{\bm 1}^{\top}$ is still in the left kernel of $\Omega_L$.)
    Hence the dimension of the left kernel of $\Omega_L$ is $d+1$.
    Note that the matrix
    \begin{equation*}
        M := \begin{bmatrix} P & L \\ {\bm 1}^{\top} & {\bm 0}^{\top} \end{bmatrix}
    \end{equation*}
    is row independent as $L$ is nonsingular.
    Hence the row vectors of $M$ span the left kernel of $\Omega_L$.
    Since this holds for any KKT point $(p, L)$, any KKT points can be transformed to each other by an affine transformation. 
\end{proof}

We are now ready to prove the key theorem of the section.

\begin{proof}[Proof of \Cref{thm:3}]
Since the energy function as well as the volume is invariant under isometries,  
we may restrict our attention to realizations $(p,L)$ with 
$p(v_1)=\mathbf{0}$ and $\det(L)\geq 1$, and in the subsequent discussion any realization is assumed to satisfy these properties.

By \Cref{lem:vol2}, $\prob$ has a global minimizer $(p^*, L^*)$. 
As $\det L^*\geq 1$, the gradient of $\log \det(\cdot)$ at $(p^*, L^*)$ is nonzero, which means that the assumption for the KKT theorem holds at $(p^*, L^*)$.
Hence, by the KKT theorem, $(p^*, L^*)$ is a KKT point with multiplier $\lambda^*$.

Take any KKT point $(p,L)$ with multiplier $\lambda$, 
and let $P^*$ and $P$ be the coordinate matrices of $p^*$ and $p$, respectively.
By \Cref{lem:vol4}, we have $\det(L)=1=\det(L^*)$.
By \Cref{lem:vol5}, $(G,p^*,L^*)$ is an affine image of $(G,p,L)$.
Since $p (v_1)=p^* (v_1)= \mathbf{0}$,  this in turn implies  $[P^*\ L^*]=A^*[P\ L]$ for some  $A\in \mathbb{R}^{d\times d}$.

By $\det(L)=\det(L^*)=\det(A^*L)$, we have $\det(A^*)=1$.
Let $\mu_1,\dots, \mu_d$ be the eigenvalues of $A^*A^{*\top}$.
Then we have 
\begin{align*}
{\cal E}_{G,\omega}(p^*, L^*) 
&= {\rm Tr} \left( \begin{bmatrix} P^* & L^* \end{bmatrix} \mathcal{L}_{\mathbb{Z}^d}(G,\omega) \begin{bmatrix}P^* & L^*\end{bmatrix}^{\top} \right) \\ 
&={\rm Tr}\left( A^*\begin{bmatrix}P & L\end{bmatrix}\mathcal{L}_{\mathbb{Z}^d}(G,\omega) \begin{bmatrix}P & L\end{bmatrix}^{\top} A^{*\top}\right) \\ 
&=\lambda{\rm Tr}\left( A^* A^{*\top}\right) & \text{(by \Cref{lem:vol4})}\\ 
&=\lambda(\mu_1+\dots+\mu_d) \\
&\geq \lambda(d \sqrt[d]{\mu_1\dots \mu_d}) & \text{(by the AM-GM inequality)} \\
&=\lambda d & \text{(by $\det(A^*)=1$)} \\
&={\rm Tr} \left( \begin{bmatrix} P & L \end{bmatrix} \mathcal{L}_{\mathbb{Z}^d}(G,\omega) \begin{bmatrix}P & L\end{bmatrix}^{\top} \right) & \text{(by \Cref{lem:vol4})} \\
&={\cal E}_{G,\omega}(p,L). 
\end{align*}
Since ${\cal E}_{G,\omega}(p,L)\geq {\cal E}_{G,\omega}(p^*, L^*)$, the equality holds throughout, and $(p,L)$ is also a global minimizer.
By the equality condition for the AM-GM inequality above, we have $\mu_1=\mu_2=\dots=\mu_d$, 
which together with $\det(A^*)=1$ implies that $A^*A^{*\top}$ is the identity matrix. 
In other words, $A^*$ is orthogonal, and $(G,p^*,L^*)$ is congruent to $(G,p,L)$.
\end{proof}

\section*{Acknowledgements} 
This project was initiated during the workshop on Rigidity and Flexibility of Microstructures held at the American Institute of Mathematics (San Jose, California) in November 2019. We also gratefully acknowledge the Fields Institute for Research in Mathematical Sciences for support during the Thematic Program on Geometric Constraint Systems, Framework Rigidity, and Distance Geometry (January 1 - June 30, 2021) and the Focus Program on Geometric Constraint Systems (July 1 - August 31, 2023). We thank Robert Connelly and Steven (Shlomo) Gortler for helpful discussions.

Sean Dewar was supported in part by the Heilbronn Institute for Mathematical Research.
Bernd Schulze was supported in part by UK Research and Innovation (grant number UKRI2397), under the EPSRC Mathematical Sciences Small Grant scheme. 
Shin-ichi Tanigawa was supported by JST CREST (grant number JPMJCR24Q2) and JSPS KAKENHI (grant number 24K14835).
Louis Theran was supported in part by UK Research and Innovation (grant number UKRI1112), under the EPSRC Mathematical Sciences Small Grant scheme.

\bibliographystyle{abbrvurl}
\bibliography{lit}

\begin{appendices}

\section{Proof of \texorpdfstring{\Cref{l:connelly lemma}}{integral lemma}}
\label{app:int}

In this appendix we prove the following technical lemma.

\connellylemma*

We recall the following terminology.
Let $S \subset \mathbb{R}$ be a finite algebraically independent set over the rationals. We denote the real closure of  $\mathbb{Q}(S)$ by $K$. 
We define a polynomial map $f:\mathbb{R}^m \rightarrow \mathbb{R}^n$ to be \emph{$S$-integral} if all coefficients of its coordinate polynomials lie in $K$. A (real) semi-algebraic set is \emph{$S$-integral} if it can be defined by a Boolean combination of polynomial equations and inequalities with coefficients in $K$. 
In the case where $S = \emptyset$, we just say \emph{integral} in both cases.

We require the following variation of the famous Tarski-Seidenberg theorem.

\begin{theorem}[Tarski-Seidenberg theorem]\label{thm:ts}
    Let $A \subset \mathbb{R}^{n+1}$ be an $S$-integral semi-algebraic set.
    Then the projection of $A$ onto its first $n$ coordinates is also an $S$-integral semi-algebraic set.
\end{theorem}

\begin{lemma}\label{l:semialg0}
    The image or preimage of an $S$-integral semi-algebraic set by an $S$-integral polynomial map is also $S$-integral.
\end{lemma}

\begin{proof}
    Suppose $A \subset \mathbb{R}^n$ is $S$-integral. Then there exist polynomials 
$g_1,\dots,g_r \in K[y_1,\dots,y_n]$ and a Boolean formula $\Phi$ such that 
\[
A = \{ y \in \mathbb{R}^n : \Phi(g_1(y),\dots,g_r(y)) \}.
\]
If $f:\mathbb{R}^m \to \mathbb{R}^n$ is an $S$-integral polynomial map, then
\[
f^{-1}(A) = \{ x \in \mathbb{R}^m : f(x) \in A \}
= \{ x \in \mathbb{R}^m : \Phi(g_1(f(x)),\dots,g_r(f(x))) \}.
\]
Each $g_i \circ f$ is again a polynomial with coefficients in $K$,
so $f^{-1}(A)$ is $S$-integral.

Now suppose $B \subseteq \mathbb{R}^m$ is $S$-integral.  
Define
\[
C = \{ (x,y) \in \mathbb{R}^m \times \mathbb{R}^n : x \in B, \; y = f(x) \}.
\]
Then $C$ is $S$-integral semi-algebraic, since it is given by the defining inequalities of $B$ together with the polynomial equations $y_i = f_i(x)$.  
The image is \(f(B) = \pi_y(C)\),
where $\pi_y : \mathbb{R}^m \times \mathbb{R}^n \to \mathbb{R}^n$ is the projection.  
By the Tarski--Seidenberg theorem, $\pi_y(C)$ is an $S$-integral semi-algebraic set.  
Therefore $f(B)$ is $S$-integral.
\end{proof}

A point $x=(x_1,\ldots,x_m) \in \mathbb{R}^m$ is \emph{$S$-generic} if its coordinates are pairwise distinct and  algebraically independent over $\mathbb{Q}(S)$.
Equivalently, $x$ is $S$-generic if and only if its coordinates are pairwise distinct and  algebraically independent over $K$.
For a finite set $S$, the set of $S$-generic points is conull in $\mathbb{R}^m$.
Again, if $S= \emptyset$ then we just say a point is \emph{generic}.

\begin{lemma}\label{l:semialg1}
    An $S$-integral semi-algebraic set contains an $S$-integral generic point if and only if it has non-empty interior in the Euclidean metric topology.
\end{lemma}

\begin{proof}
    If the set has nonempty interior, it contains an open ball (in the metric topology). Since $\mathbb{Q}(S)$ is countable, there is a countable number of nonzero polynomials with coefficients in $\mathbb{Q}(S)$. Therefore, the complement of the $S$-generic points is a countable union of proper algebraic varieties over $\mathbb{Q}(S)$, each of which is of (Lebesgue) measure zero. Hence the set of non-$S$-generic points has measure zero, which implies that $S$-generic points are dense and occur in every open ball.  

    Conversely, if an $S$-integral semi-algebraic set has empty interior then it is contained in finitely many proper algebraic sets, each being $S$-integral. But an $S$-generic point is not in any proper $S$-integral algebraic set by definition. 
\end{proof}

\begin{lemma}\label{l:semialg2}
    Let $f:\mathbb{R}^m \rightarrow \mathbb{R}^n$ be a $S$-integral polynomial map for some finite set $S$.
    If $y\in \mathbb{R}^m$ is $S$-generic and $k= \max_{x \in \mathbb{R}^m} \rank df(x)$,
    then $\rank df(z) = k$ if $f(y)=f(z)$.
\end{lemma}

\begin{proof}
    Define the proper $S$-integral algebraic set $C = \{ x \in \mathbb{R}^m : \rank df(x) < k\}$.
    By \Cref{l:semialg0},
    $f^{-1}(f(C))$ is an $S$-integral semi-algebraic set.
    Suppose for contradiction that $f^{-1}(f(C))$ contains a non-empty open subset.
    Since $\mathbb{R}^m \setminus C$ is an open dense set, there exists a non-empty open set $U \subset f^{-1}(f(C))$ where $U \cap C = \emptyset$.
    By the constant rank theorem (e.g., \cite[Theorem 9.32]{rudin}), the (real) Zariski closure of $f(C)$ is $k$-dimensional.
    However,
    by the real semi-algebraic version of Sard's theorem (e.g., \cite[Theorem 9.6.2]{BochnakCosteRoy}),
    this set must have dimension at most $k-1$, a contradiction.
    Hence $f^{-1}(f(C))$ has an empty interior and $y \notin f^{-1}(f(C))$ by \Cref{l:semialg1}.
\end{proof}

\begin{proof}[Proof of \Cref{l:connelly lemma}]
   Suppose for a contradiction that there exists $z$ where $f(z) = f(y)$ and $\ker df(z)^\top \neq \ker df(y)^\top$.
    By \Cref{l:semialg2} and the constant rank theorem,
    it follows that for sufficiently small open neighbourhoods $U_y,U_z$ of $y,z$ respectively,
    we have that the set $f(U_y) \cap f(U_z)$ has an empty interior in $f(U_y)$.
    By choosing $U_y$ to be suitably small, we will also have that the set $A := (f^{-1}(f(U_y) \cap f(U_z))) \cap U_y$ has empty interior and contains $y$.
   We now choose $U_y$ and $U_z$ to also be $S$-integral semi-algebraic sets by defining them by suitable $S$-integral linear inequalities.
    By \Cref{l:semialg0},
   the set $A$ is an $S$-integral semi-algebraic set that contains $y$,
    contradicting \Cref{l:semialg1}.
    \end{proof}

\section{From finite to periodic super stability}\label{appb}
Given a super stable finite tensegrity, it is natural to wonder if it can be converted to a super stable $\mathbb{Z}^d$-tensegrity. We show that this is always possible.

Given a $d$-dimensional finite framework (or tensegrity) $(G,p)$,
we consider the following construction of a $\mathbb{Z}^d$-framework $(G^{\circ},p^{\circ},L)$ (or a $\mathbb{Z}^d$-tensegrity by keeping the edge types):
\begin{itemize}
\item We pick $d$ pairs $(u_1,v_1), (u_2,v_2), \dots, (u_d,v_d)$ of vertices such that 
\begin{equation}\label{eq:construction1}
\text{$u_i\neq v_j$ for any $1\leq i,j\leq d$ and $v_i\neq v_j$ for any $1\leq i,j\leq d$ with $i\neq j$}.
\end{equation}
(Note that $u_i=u_j$ may hold and $u_iv_i$ may not be an edge of $G$.)
We further impose the condition that 
\begin{equation}\label{eq:construction2}
\text{$\{p(v_i)-p(u_i): 1\leq i\leq d\}$ is linearly independent.}
\end{equation}

\item We define a $\mathbb{Z}^d$-gain graph $G^{\circ}= (V^{\circ}, E^{\circ})$ from $G=(V,E)$ by the following procedure:
\begin{itemize}
    \item Let $V'=\{v_1,\dots, v_d\}$.
    Also, let $e_i$ be the unit vector in $\mathbb{R}^d$ whose $i$-th coordinate is one.
    \item If an edge $e=ab\in E$ satisfies  
    $a=v_i$ and $b=v_j$ with $i<j$,
    then assign a direction from $a$ to $b$ with the edge label $e_j-e_i$.
    \item If an edge $e=ab\in E$ satisfies  
    $a\in V\setminus V'$ and $b=v_i$,
    then assign a direction from $a$ to $b$ with the edge label $e_i$.
    \item If an edge $e=ab\in E$ satisfies  
    $a, b\in V\setminus V'$,
    then assign a direction from $a$ to $b$ with the trivial edge label.
    \item Finally, identify $u_i$ and $v_i$ for all $i=1,\dots, d$ keeping all directed edges (including self-loops). The resulting vertex set is $V^{\circ}=V\setminus V'$.
\end{itemize}
\item Let $L$ be the $d\times d$ matrix whose $i$-th column is $p(v_i)-p(u_i)$. By \cref{eq:construction2}, $L$ is non-singular.
\item Let $p^{\circ}$ be the restriction of $p$ to $V^{\circ}$.
\end{itemize}
We now call $(G^{\circ},p^{\circ},L)$ the 
{\em associated $\mathbb{Z}^d$-framework of $(G,p)$ with respect to $(u_1,v_1),\dots, (u_d,v_d)$} (or the {\em associated $\mathbb{Z}^d$-tensegrity} if $(G,p)$ is a tensegrity, keeping the edge types).

Let us look at an example.
Consider a 2-dimensional tensegrity $(G,p)$ with 
8 vertices obtained from the regular octagon by adding four disjoint diagonals as shown in \Cref{fig:octagon},
where $G$ is given by
$V(G)=\{0,1,2,3,4,5,6,7\}$ and 
\[
E(G)=\{(0,1), (1,2), (2,3), (3,4), (4,5), (5,6),(6,7),(7,0),(0,3),(4,7),(1,6),(2,5)\}.\]
We pick two pairs $(u_1,v_1)=(0,4)$ and $(u_2,v_2)=(2,6)$ of vertices.
Then, the associated $\mathbb{Z}^2$-tensegrity $(G^{\circ}, p^{\circ}, L)$ is given by
$V(G^{\circ})=\{0,1,2,3,5,7\}$ and
\[
\begin{split}E(G^{\circ})=&\{(0,1;(0,0)), (1,2;(0,0)), (2,3;(0,0)), (3,0;e_1), (5,0;e_1),(5,2;e_2),(7,2;e_2),(7,0;(0,0)),\\
&(0,3;(0,0)),(7,0;e_1),(1,2;e_2),(2,5;(0,0))\}
\end{split}
\]
and $L=\begin{pmatrix}2 & 0 \\ 0 & 2\end{pmatrix}$.
The covering of $(G^{\circ}, p^{\circ}, L)$ is as shown in \Cref{fig:octagon}.

\begin{figure}[ht]
\centering
\begin{minipage}{0.35\textwidth}
\includegraphics[scale=0.8]{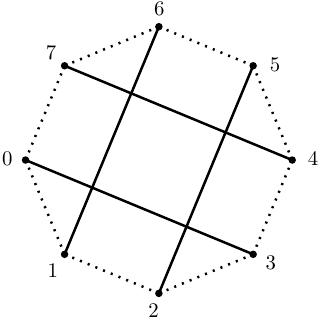}
\end{minipage}
\begin{minipage}{0.44\textwidth}
\includegraphics[scale=0.4]{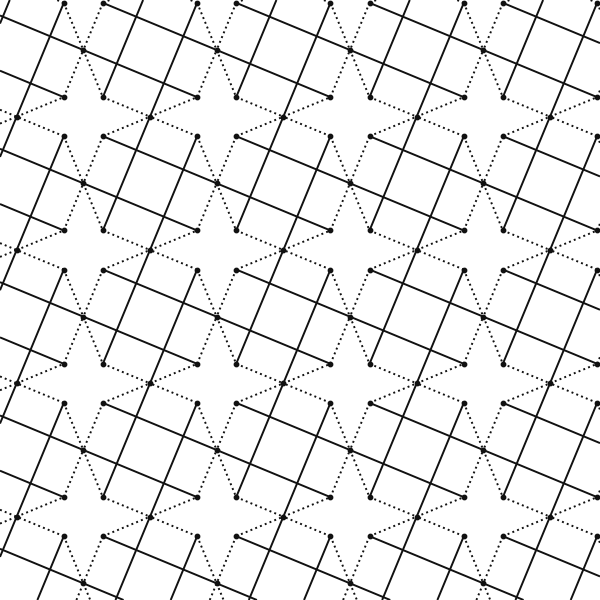}
\end{minipage}
\caption{
Shown on the left is a 2-dimensional finite octagon tensegrity with added diagonals, where cables are shown dotted and struts solid. This tensegrity is super stable and can be used to construct an associated super stable $\mathbb{Z}^2$-tensegrity whose covering is shown on the right.}
\label{fig:octagon}
\end{figure}

The finite octagon tensegrity $(G,p)$ in \Cref{fig:octagon} is super stable with equilibrium stress
\[
\begin{aligned}
\omega(01) &= 2 + \sqrt{2}, 
\omega(12) = \sqrt{2} + 1, 
\omega(23) = 2 + \sqrt{2}, 
\omega(34) = \sqrt{2} + 1, \\
\omega(45) &= 2 + \sqrt{2}, 
\omega(56) = \sqrt{2} + 1, 
\omega(67) = 2 + \sqrt{2}, 
\omega(70) = \sqrt{2} + 1, \\
\omega(03) &= -1, 
\omega(47) = -1, 
\omega(16) = -1, 
\omega(25) = -1.
\end{aligned}
\] 
Indeed, the corresponding weighted Laplacian $\mathcal{L}(G,\omega)$ is 
\[
\scalemath{0.8}{
\begin{bmatrix}
2\sqrt{2}+2 & -\sqrt{2}-2 & 0 & 1 & 0 & 0 & 0 & -1-\sqrt{2} \\
-\sqrt{2}-2 & 2\sqrt{2}+2 & -1-\sqrt{2} & 0 & 0 & 0 & 1 & 0 \\
0 & -1-\sqrt{2} & 2\sqrt{2}+2 & -\sqrt{2}-2 & 0 & 1 & 0 & 0 \\
1 & 0 & -\sqrt{2}-2 & 2\sqrt{2}+2 & -1-\sqrt{2} & 0 & 0 & 0 \\
0 & 0 & 0 & -1-\sqrt{2} & 2\sqrt{2}+2 & -\sqrt{2}-2 & 0 & 1 \\
0 & 0 & 1 & 0 & -\sqrt{2}-2 & 2\sqrt{2}+2 & -1-\sqrt{2} & 0 \\
0 & 1 & 0 & 0 & 0 & -1-\sqrt{2} & 2\sqrt{2}+2 & -\sqrt{2}-2 \\
-1-\sqrt{2} & 0 & 0 & 0 & 1 & 0 & -\sqrt{2}-2 & 2\sqrt{2}+2
\end{bmatrix}
}
\]
where the rows and the columns are ordered by $0,1,2,3,4,5,6,7$,
and it  has rank five and is positive semidefinite.

Since there is a one-to-one correspondence between $E(G)$ and $E(G^{\circ})$, we may consider $\omega$ as an edge-weight of $G^{\circ}$. 
Then 
the corresponding weighted $\mathbb{Z}^2$-Laplacian $\mathcal{L}_{\mathbb{Z}^2}(G^{\circ},\omega)$ is
\[
\scalemath{0.8}{
\begin{bmatrix}
-4\sqrt{2}-4 & \sqrt{2}+2 & 0 & \sqrt{2} & \sqrt{2}+2 & \sqrt{2} & -2\sqrt{2}-2 & 0 \\
\sqrt{2}+2 & -2\sqrt{2}-2 & \sqrt{2} & 0 & 0 & 0 & 0 & -1 \\
0 & \sqrt{2} & -4\sqrt{2}-4 & \sqrt{2}+2 & \sqrt{2} & \sqrt{2}+2 & 0 & -2\sqrt{2}-2 \\
\sqrt{2} & 0 & \sqrt{2}+2 & -2\sqrt{2}-2 & 0 & 0 & 1+\sqrt{2} & 0 \\
\sqrt{2}+2 & 0 & \sqrt{2} & 0 & -2\sqrt{2}-2 & 0 & \sqrt{2}+2 & 1+\sqrt{2} \\
\sqrt{2} & 0 & \sqrt{2}+2 & 0 & 0 & -2\sqrt{2}-2 & -1 & \sqrt{2}+2 \\
-2\sqrt{2}-2 & 0 & 0 & 1+\sqrt{2} & \sqrt{2}+2 & -1 & -2\sqrt{2}-2 & 0 \\
0 & -1 & -2\sqrt{2}-2 & 0 & 1+\sqrt{2} & \sqrt{2}+2 & 0 & -2\sqrt{2}-2
\end{bmatrix}
}
\]
where the rows and the columns are ordered by $0,1,2,3,5,7,\ell_x,\ell_y$.
Let $\tau$ be the non-singular matrix which subtracts the row of $\ell_x$ from the row of $0$ and subtracts the row of $\ell_y$ from the row of $2$,
and let $\rho$ be the permutation matrix that converts the ordering of the indices from $0,1,2,3,5,7,\ell_x,\ell_y$ to $0,1,2,3,\ell_x,5,\ell_y,7$.
Then, we have 
\begin{equation}\label{eq:construction3}
\mathcal{L}(G,\omega)=\rho \tau \mathcal{L}_{\mathbb{Z}^2}(G^{\circ},\omega) \tau^\top \rho^\top.
\end{equation}
Since $\mathcal{L}(G,\omega)$ has rank five and is positive semidefinite with $P \mathcal{L}(G,\omega)= \mathbf{0}_{d \times |V|}$, \cref{eq:construction3} implies that $\mathcal{L}_{\mathbb{Z}^2}(G^{\circ},\omega)$ has rank five and is positive semidefinite and $[P^{\circ},L]\mathcal{L}_{\mathbb{Z}^2}(G^{\circ},\omega)= \mathbf{0}_{d \times |V|}$.
As the set of edge directions for $(G,p)$ do not lie on a conic at infinity, the set of edge directions for $(G^{\circ},p^{\circ},L)$ do not lie on a conic at infinity.

This argument works in general and we have the following.
\begin{proposition}
Let $(G,p)$ be a $d$-dimensional tensegrity,
$(u_1,v_1),\dots, (u_d,v_d)$ be pairs of vertices satisfying 
\cref{eq:construction1} and \cref{eq:construction2},
and $(G^{\circ},p^{\circ},L)$ be the associated $\mathbb{Z}^d$-tensegrity.
If $(G,p)$ is super stable, then 
$(G^{\circ},p^{\circ},L)$ is super stable.
\end{proposition}

\end{appendices}
\end{document}